\documentclass[12pt]{amsart}

\usepackage{amsmath,amsthm,amssymb,amsfonts,bbm,latexsym,xcolor,mathrsfs,enumerate,array}

\usepackage[hidelinks]{hyperref}

\allowdisplaybreaks

\numberwithin{equation}{section}

\title[Twisted moments of characteristic polynomials]{Twisted moments of characteristic polynomials of random matrices in the unitary group
}
\author{Siegfred Baluyot}
\address{East Carolina University, East Fifth Street Greenville, NC 27858, USA}
\email{\href{mailto:baluyots24@ecu.edu}{baluyots24@ecu.edu}}
\author{Brian Conrey}
\address{American Institute of Mathematics \\
Caltech 8-32 1200 E California Blvd Pasadena, CA 91125}
\email{\href{mailto:conrey@aimath.org}{conrey@aimath.org}}


\thanks{SB was partially supported by NSF DMS-1854398 FRG. BC is partially supported by a grant from the NSF}

\newtheorem{theorem}{Theorem}[section]

\newtheorem{lemma}[theorem]{Lemma}

\theoremstyle{definition}
\newtheorem*{definition}{Definitions}

\begin{document}

\begin{abstract}
Recently, Keating and the second author of this paper devised a heuristic for predicting asymptotic formulas for moments of the Riemann zeta-function $\zeta(s)$. Their approach indicates how lower twisted moments of $\zeta(s)$ may be used to evaluate higher moments. In this paper, we present a rigorous random matrix theory analogue of their heuristic. To do this, we develop a notion of ``twisted moment'' of characteristic polynomials of matrices in the unitary group $U(N)$, and we prove several identities involving Schur polynomials. Our results may be viewed as a proof of concept of the heuristic for $\zeta(s)$.
\end{abstract}

\maketitle

\section{Introduction and results}

In this paper, we present a rigorous random matrix theory analogue of the heuristic in \cite{ck5}, which was devised by Keating and the second author of this paper as an approach towards finding asymptotic formulas for moments of the Riemann zeta-function. Our results may be viewed as a proof of concept of their heuristic.

Since the monumental discovery of Dyson and Montgomery~\cite{dyson,montgomery}, random matrix theory has become a useful predictive tool in the study of $L$-functions. Keating and Snaith, in their seminal work~\cite{ks1}, applied random matrix theory to predict the leading terms of asymptotic formulas for the moments
\begin{equation}\label{eqn: zetamoment}
\int_0^T|\zeta(\tfrac{1}{2}+it)|^{2r} \,dt
\end{equation}
of the Riemann zeta function $\zeta(s)$ when $r$ is any fixed complex number with real part $>-1/2$. At that time, number theoretic heuristics involving long Dirichlet polynomials have resulted in predictions for the leading term for $r=3,4$ but failed to give a feasible prediction for $r=5$~\cite{conreyghosh,conreygonek}. To date, asymptotic formulas for \eqref{eqn: zetamoment} can be proved only for $r=0$ (trivially), $r=1$, and $r=2$~\cite{titchmarsh}. Remarkably, the predictions of Keating and Snaith agree with the known formulas and the long Dirichlet polynomial predictions for $r=0,1,2,3,4$.

Today, there are several different heuristics that predict precise asymptotic formulas for \eqref{eqn: zetamoment} and agree with the conjectures of Keating and Snaith. One of these heuristics is the \textit{recipe} in \cite{cfkrs}, which also predicts lower order terms and applies to a general family of $L$-functions. While the recipe seems to give the correct answer, its individual steps may be unjustifiable. This is because some steps in the recipe may throw away terms that are of the same size as the main term, while other steps bring back other terms of the same size~\cite[Section~2.1]{cfkrs}. Curiously, these large errors in the individual steps of the recipe seem to cancel out to give the correct answer. 

In order to investigate how and why the recipe works, Keating and the second author of this paper revisited the long Dirichlet polynomial approach. In a series of papers~\cite{ck1,ck2,ck3,ck4,ck5}, they introduced a new heuristic that recovers the recipe prediction for averages of Dirichlet polynomial approximations of products of shifted zeta-functions. These averages are given by
\begin{equation}\label{eqn: dirichletpolynomialapprox}
\frac{1}{T}\int_0^{\infty}\psi\left( \frac{t}{T}\right) \sum_{m\leq X} \frac{\tau_A (m) }{m^{\frac{1}{2}+it}}  \sum_{n\leq X} \frac{\tau_B (n) }{n^{\frac{1}{2}-it}} \,dt,
\end{equation}
where $\psi$ is a smooth function with compact support, $X$ is a parameter tending to $\infty$, $A$ and $B$ are finite multisets of small complex numbers, and the coefficients $\tau_A$ are defined for finite multisets $A$ by
\begin{equation*}
\sum_{m=1}^{\infty}\frac{\tau_A(m)}{m^s} =\prod_{\alpha\in A} \zeta(s+\alpha),
\end{equation*}
where the product counts multiplicity. Hence, \eqref{eqn: dirichletpolynomialapprox} is an approximation to the shifted moment
\begin{equation*}
\frac{1}{T}\int_0^{\infty}\psi\left( \frac{t}{T}\right) \prod_{\alpha\in A} \zeta(\tfrac{1}{2}+\alpha+it)  \prod_{\beta\in A} \zeta(\tfrac{1}{2}+\beta-it) \,dt.
\end{equation*}
To keep our discussion of the ideas in \cite{ck5} concise, let us restrict our attention to the special case in which $A$ and $B$ are sets. For each integer $\ell$ with $1\leq \ell\leq \min\{|A|,|B|\}$, we may partition $A$ into $\ell$ nonempty sets to write $A=A_1\cup\cdots \cup A_{\ell}$ and express $\tau_A$ as the Dirichlet convolution $\tau_A=\tau_{A_1}*\cdots *\tau_{A_{\ell}}$. We may do the same for $\tau_B$ and thus write \eqref{eqn: dirichletpolynomialapprox} as
\begin{align}
\frac{1}{T}\int_0^{\infty} \psi\left( \frac{t}{T}\right)
& \sum_{m_1 m_2 \cdots m_{\ell} \leq X}  \frac{\tau_{A_1} (m_1)\tau_{A_2} (m_2)\cdots \tau_{A_{\ell}} (m_{\ell})   }{(m_1 m_2\cdots m_{\ell} )^{\frac{1}{2}+it}  } \notag \\
& \times \sum_{n_1 n_2 \cdots n_{\ell} \leq X} \frac{  \tau_{B_1} (n_1)\tau_{B_2} (n_2)\cdots \tau_{B_{\ell}} (n_{\ell}) }{  (n_1 n_2 \cdots n_{\ell})^{\frac{1}{2}-it}} \,dt.  \label{eqn: dirichletpolynomialapprox2}
\end{align}
The basic idea behind the approach in \cite{ck5} is to introduce positive integer parameters $M_j,N_j$ for $j=1,2,\dots,\ell$, collect the terms in \eqref{eqn: dirichletpolynomialapprox2} that have $m_j/n_j$ close to $M_j/N_j$ for all $j$, and then sum over all possible values of the $M_j,N_j$ such that $M_1M_2\cdots M_{\ell}= N_1N_2\cdots N_{\ell}$ and $(M_j,N_j)=1$ for all $j$. This involves summing the terms with $m_jN_j-n_jM_j=h_j$, where $h_j$ is an integer variable, and then summing over the variables $h_j$. This heuristic was inspired by the work of Bogomolny and Keating~\cite{bk1,bk2}, who used a similar decomposition of sums over primes to study the $n$-point correlation of zeros of $\zeta(s)$. To evaluate the sums over the $M_j,N_j$, the approach in \cite{ck5} uses Euler products and the key identity
\begin{equation}\label{eqn: unitaryidentityzeta}
\sum_{m=0}^{\infty} \frac{\tau_A(p^m)\tau_B(p^m)}{p^m} = \sum_{\substack{m_1+m_2+\cdots +m_{\ell}  = n_1+n_2+\cdots +n_{\ell} \\ \min\{m_j,n_j\} = 0  \ \ \forall j}} \ \prod_{j=1}^{\ell} \Bigg( \sum_{r=0}^{\infty} \frac{\tau_{A_{j}}(p^{r+m_j}) \tau_{B_{j}}(p^{r+n_j}) }{p^{ r + \frac{1}{2}m_j +\frac{1}{2} n_j} }\Bigg) 
\end{equation}
(see Lemma~3 of \cite{ck5} and its refinement in the proof of Theorem~4.4 of \cite{bc}).

The condition $m_jN_j-n_jM_j=h_j$ arises when estimating (or attempting to estimate) the twisted moment
\begin{equation*}
\frac{1}{T}\int_0^{\infty}\psi\left( \frac{t}{T}\right) \left( \frac{M}{N}\right)^{it} \prod_{\alpha\in A_j} \zeta(\tfrac{1}{2}+\alpha+it)  \prod_{\beta\in B_j} \zeta(\tfrac{1}{2}+\beta-it)\,dt
\end{equation*}
via the standard approach of using Dirichlet polynomial approximations or approximate functional equations (see, for example, \cite{hughesyoung}). Therefore, the heuristic in \cite{ck5} essentially combines lower twisted moments to find an asymptotic formula for higher (untwisted) moments.

In this paper, we show that this can be done rigorously with the Keating-Snaith random matrix model for moments of $\zeta(s)$. To do this, we introduce a notion of ``twisted moment'' of characteristic polynomials of matrices in the unitary group $U(N)$ (equation \eqref{eqn: twistedmomentdefinition} below). With our definition of twisted moment, we show that (untwisted) moments of characteristic polynomials of matrices in $U(N)$ can be expressed as sums of products of lower twisted moments (Theorem~\ref{thm: BK} below). We also evaluate these twisted moments (Theorem~\ref{thm: twistedrecipe} below), thus generalizing the formula for (untwisted) moments due to Farmer, Keating, Rubinstein, Snaith, and the second author of this paper~\cite{cfkrs2,ks1}. We use Theorem~\ref{thm: BK} and Theorem~\ref{thm: twistedrecipe} for lower twisted moments to prove the formula for higher (untwisted) moments in Section~\ref{sec: recipe} below, thereby giving a proof of concept of the heuristic in \cite{ck5}.

We now define the twisted moments of characteristic polynomials of matrices in $U(N)$.

\begin{definition}
If $n$ is any positive integer, then a \textit{dominant weight of length $n$} is an $n$-tuple $\lambda=(\lambda_1,\lambda_2,\dots,\lambda_n)$ of integers such that
\begin{equation*}
\lambda_1 \geq \lambda_2 \geq \dots \geq \lambda_n.
\end{equation*}
For a dominant weight $\lambda=(\lambda_1,\lambda_2,\dots,\lambda_n)$ of length $n$, we define the \textit{generalized Schur function} $s_{\lambda}$ of $n$ variables $x_1,x_2,\dots,x_n$ by
\begin{equation}\label{eqn: genschurdef}
s_{\lambda}(x_1,x_2,\dots,x_n):= \frac{\det\Big(x_j^{n-k+\lambda_k} \Big)_{1\leq j,k\leq n} }{ \det\Big(x_j^{n-k} \Big)_{1\leq j,k\leq n} }.
\end{equation}
If $N$ is a positive integer, $A$ and $B$ are finite multisets of complex numbers, and $\lambda$ is a dominant weight of length $N$, then we define the \textit{twisted moment} $\mathcal{M}_N(A,B;\lambda)$ by
\begin{equation}\label{eqn: twistedmomentdefinition}
\mathcal{M}_N(A,B;\lambda) : = \int_{U(N)} \prod_{\alpha \in A} \det(I-\alpha g^*) \prod_{\beta \in B} \det(I-\beta g) s_{\lambda} (t_1,t_2,\dots,t_N) \,dg,
\end{equation}
where the products count multiplicity, the integral is over the unitary group $U(N)$ equipped with the Haar measure $dg$ that is normalized so that $U(N)$ has volume $1$, $I$ is the $N\times N$ identity matrix, $g^*$ is the conjugate transpose of $g$, and the numbers $t_1,t_2,\dots,t_N$ are the eigenvalues of $g$, each repeated according to its multiplicity.
\end{definition}

A few comments about these definitions are in order. Bump and Gamburd~\cite[p. 232]{bumpgamburd} use the term ``dominant weight'' to refer to finite decreasing sequences of integers, and we emulate their use of the term. Every partition in symmetric function theory is a dominant weight when written as a finite sequence, i.e., when we do not include an infinite number of zeros in the definition of the partition. The standard definition of length of a partition is the number of its positive parts, and so its length as a dominant weight may be different from its length as a partition. However, we will always use the latter when referring to the length of a partition. If the dominant weight $\lambda = (\lambda_1,\lambda_2,\dots,\lambda_n)$ is a partition, i.e., if it satisfies $\lambda_n\geq 0$, then the generalized Schur function $s_{\lambda}(x_1,x_2,\dots,x_n)$ is precisely the Schur polynomial in the variables $x_1,x_2,\dots,x_n$ that is indexed by $\lambda$, because \eqref{eqn: genschurdef} is Jacobi's bialernant formula that defines the Schur polynomial in this case (see, for example, Chapter I, \S3 of \cite{macdonald}). We adopt standard conventions in symmetric function theory and consider two partitions to be the same if they differ by only a string of zeros at the end of the sequence. Thus, for example, $(3,2)$, $(3,2,0,0)$, and $(3,2,0,0,\dots)$ are the same partition. We also adhere to the convention that if $\mu$ is a partition of length greater than a positive integer $m$, then $s_{\mu}(y_1,y_2,\dots,y_m)$ is defined to be zero for any $y_1,y_2,\dots,y_m$. On the other hand, every generalized Schur function in this paper that is not a Schur polynomial is indexed by a dominant weight of length equal to the number of variables in its argument. From the definition \eqref{eqn: genschurdef} and properties of the determinant, we see that generalized Schur functions $s_{\lambda}(x_1,x_2,\dots,x_n)$ are invariant under permutations of the variables $x_1,x_2,\dots,x_n$. Thus, the twisted moment \eqref{eqn: twistedmomentdefinition} does not depend on the order of the eigenvalues $t_j$. The (untwisted) moment that models moments of $\zeta(s)$ (see, for example, \cite{ks1} or \cite{cfkrs2}) is $\mathcal{M}_N(A,B;0)$, where we use $0$ to denote the zero partition $(0,0,\dots)$.

Having defined the twisted moment \eqref{eqn: twistedmomentdefinition}, we may now state our main device in carrying out a rigorous random matrix theory analogue of the heuristic in \cite{ck5}. We call this device the ``BK~Splitting'' after Bogomolny and Keating, whose works~\cite{bk1,bk2} inspired \cite{ck1,ck2,ck3,ck4,ck5} and in turn this paper. Recall that if $C$ and $D$ are multisets, then the multiset sum $C+D$ is the multiset such that the multiplicity of each $e$ in $C+D$ is equal to the multiplicity of $e$ in $C$ plus the multiplicity of $e$ in $D$.

\begin{theorem}[The BK Splitting]\label{thm: BK}
Let $N$ be a positive integer, and suppose that $A,B$ are finite multisets of complex numbers. Suppose that $A$ equals the multiset sum $A=A_1+ A_2$ and that $B$ equals the multiset sum $B=B_1 + B_2$. Let $\mathcal{M}_N$ be defined by \eqref{eqn: twistedmomentdefinition}. Then
$$
\mathcal{M}_N(A,B;0) = \sum_{\lambda} \mathcal{M}_N(A_1,B_1;\lambda) \mathcal{M}_N(A_2,B_2;-\lambda),
$$
where $0$ denotes the zero partition $(0,0,\dots)$, the sum on the right hand side is over all dominant weights $\lambda$ of length $N$, and if $\lambda=(\lambda_1,\lambda_2,\dots,\lambda_N)$ is a dominant weight, then $-\lambda$ denotes the dominant weight $(-\lambda_N,-\lambda_{N-1}\dots,-\lambda_1)$. The summand on the right-hand side is nonzero for at most finitely many $\lambda$.
\end{theorem}

The idea in our proof of the BK Splitting in Section~\ref{sec: BK} may be summarized as follows. The generalized Schur functions $s_{\lambda}$ form an orthogonal basis of the inner product space consisting of complex functions on $U(N)$ that are polynomial functions of the eigenvalues of the matrix and their reciprocals, with inner product given by $\langle F_1,F_2 \rangle =\int_{U(N)}F_1(g)\overline{F_2(g)}\,dg $, and the BK Splitting is the Parseval identity applied to the functions
\begin{equation*}
F_1(g) =\prod_{\alpha \in A_1} \det(I-\alpha g^*) \prod_{\beta \in B_1} \det(I-\beta g)
\end{equation*}
and
\begin{equation*}
F_2(g) =\prod_{\alpha \in A_2} \det(I-\alpha g^*) \prod_{\beta \in B_2} \det(I-\beta g).
\end{equation*}

To state the formula for $\mathcal{M}_N(A,B;0)$ from \cite{cfkrs2} that we will prove using the BK Splitting, we first introduce some notations. For finite multisets $A$ and $B$, define $Z(A,B)$ by
\begin{equation}\label{eqn: Zdef}
Z(A,B) =  \prod_{a\in A} \prod_{b\in B} \frac{1}{1-ab}.
\end{equation}
If $U$ and $V$ are finite multisets of complex numbers and $m$ is an integer, then we denote
\begin{equation}\label{eqn: UVNdef}
(UV)^m : = \prod_{u\in U} u^m \prod_{v\in V} v^m
\end{equation}
and
\begin{equation}\label{eqn: 1/Udef}
\frac{1}{U}:= \{ u^{-1} : u\in U\}.
\end{equation}
The products in \eqref{eqn: Zdef} and \eqref{eqn: UVNdef} count multiplicity. For ease of notation, we will state our results in the special case in which $A$ and $B$ are sets. The general results for multisets will follow immediately from analytic continuation. Using the notations \eqref{eqn: Zdef}, \eqref{eqn: UVNdef}, and \eqref{eqn: 1/Udef}, we see that equation (2.16) of \cite{cfkrs2} is equivalent to
\begin{equation}\label{eqn: recipe}
\mathcal{M}_N(A,B;0) = \sum_{ \substack{U\subseteq A \\ V\subseteq B \\ |U|=|V| } } (UV)^N Z (A\smallsetminus U \cup \tfrac{1}{V}, B\smallsetminus V \cup \tfrac{1}{U})
\end{equation}
in the special case in which $A$ and $B$ are sets of nonzero complex numbers such that $A$, $B$, $1/A$, and $1/B$ are pairwise disjoint, where the sum is over all pairs of subsets $U,V$ such that $U\subseteq A$, $V\subseteq B$, and $U$ and $V$ have the same cardinality. Note that, since \eqref{eqn: twistedmomentdefinition} is a polynomial in the elements of $A$ and $B$, it follows that \eqref{eqn: recipe} holds for any finite multisets $A$ and $B$ of complex numbers, without any restrictions, provided that the right-hand side of \eqref{eqn: recipe} is interpreted as its analytic continuation.

Our next theorem is a generalization of \eqref{eqn: recipe} to twisted moments, and allows us to prove \eqref{eqn: recipe} using the BK Splitting (Theorem~\ref{thm: BK}) and lower twisted moments, as we will see in Section~\ref{sec: recipe}.

\begin{theorem}\label{thm: twistedrecipe}
Let $N$ be a positive integer. Let $A$ and $B$ be finite sets of nonzero complex numbers such that $A$, $B$, $1/A$, and $1/B$ are pairwise disjoint. Suppose that $\lambda=(\lambda_1,\lambda_2,\dots,\lambda_N)$ is a dominant weight of length $N$. Then
\begin{equation}\label{eqn: twistedrecipe}
\mathcal{M}_N(A,B;\lambda) = \sum_{ \substack{U\subseteq A \\ V\subseteq B \\ |U|=|V| } } (UV)^N Z_{\lambda}(A\smallsetminus U \cup \tfrac{1}{V}, B\smallsetminus V \cup \tfrac{1}{U}),
\end{equation}
where $Z_{\lambda}(A,B)$ is defined by
\begin{align*}
Z_{\lambda}(A,B): = s_{\mu'}(-A)s_{\nu'}(-B) Z(A,B),
\end{align*}
where $Z(A,B)$ is defined by \eqref{eqn: Zdef}, the partitions $\mu$ and $\nu$ are defined by
\begin{equation}\label{eqn: mudef}
\mu = (\max\{\lambda_1,0\}, \max\{\lambda_2,0\}, \dots, \max\{\lambda_N,0\})
\end{equation}
and
\begin{equation}\label{eqn: nudef}
\nu = (\max\{-\lambda_N,0\}, \max\{-\lambda_{N-1},0\}, \dots, \max\{-\lambda_1,0\}),
\end{equation}
$\mu'$ and $\nu'$ denote the conjugate partitions of $\mu$ and $\nu$, respectively, and if $E=\{e_1,e_2,\dots,e_k\}$ is a multiset of complex numbers and $\rho$ is a partition, then $-E$ denotes the multiset $\{-e: e\in E\}$ and $s_{\rho}(E)$ denotes the Schur polynomial
\begin{equation}\label{eqn: schurpolynomialnotation}
s_{\rho}(E):=s_{\rho}(e_1,e_2,\dots,e_k).
\end{equation}
More generally, \eqref{eqn: twistedrecipe} holds for any finite multisets $A$ and $B$ of complex numbers, without any restrictions, provided that the right-hand side of \eqref{eqn: twistedrecipe} is interpreted as its analytic continuation.
\end{theorem}

\noindent \textit{Note}: Together with the definition \eqref{eqn: schurpolynomialnotation}, we adopt the convention that if $\emptyset$ is the empty set, then $s_{\rho}(\emptyset) =0$ for all partitions $\rho$ except the zero partition $\rho=0$, in which case $s_0(\emptyset)=1$. The definition \eqref{eqn: schurpolynomialnotation} of $s_{\rho}(E)$ is independent of the order of the $e_1,\dots,e_k$ because Schur polynomials are symmetric polynomials. Also, by the standard convention that $s_{\rho}(x_1,x_2,\dots,x_k)=0$ if $\rho$ is a partition of length strictly greater than the number $k$ of variables, we see from Theorem~\ref{thm: twistedrecipe} that $\mathcal{M}_N(A,B;\lambda) = 0$ unless $\lambda_1\leq |A|$ and $\lambda_N \geq -|B|$.

We prove Theorem~\ref{thm: twistedrecipe} in Section~\ref{sec: twistedrecipe} by generalizing a proof of \eqref{eqn: recipe} due to Bump and Gamburd~\cite[Section~3.1]{bumpgamburd} that uses properties of Schur polynomials.

Our proof of \eqref{eqn: recipe} via the BK Splitting and Theorem~\ref{thm: twistedrecipe} involves evaluating sums of products of two Schur polynomials. A key identity in our proof is given by the following theorem, which is a generalization of an identity due to Cauchy. To state it, recall the definition \eqref{eqn: Zdef} and the notations \eqref{eqn: UVNdef}, \eqref{eqn: 1/Udef}, and \eqref{eqn: schurpolynomialnotation}.

\begin{theorem}\label{thm: unitaryidentity}
Let $N$ be a positive integer. Suppose that $A,B,C,D$ are finite sets of nonzero complex numbers such that $A$, $B$, $1/A$, and $1/B$ are pairwise disjoint and $C$, $D$, $1/C$, and $1/D$ are pairwise disjoint. Suppose further that $A$ is nonempty and $B$ is nonempty. For a partition $\mu$, let $\mu_1$ denote its largest part. Then
\begin{equation}\label{eqn: unitaryidentity0}
\sum_{\substack{\mu,\nu \\ \mu_1+\nu_1\leq N}} s_{\mu}(A)s_{\mu}(B) s_{\nu}(C)s_{\nu}(D) = \Sigma_0 + \Sigma_1 + \Sigma_2,
\end{equation}
where the sum on the left-hand side is over all pairs $\mu,\nu$ of partitions such that the sum of the largest part $\mu_1$ of $\mu$ and the largest part $\nu_1$ of $\nu$ is $\leq N$, and where $\Sigma_0$ ,$\Sigma_1$, and $\Sigma_2$ are defined by
\begin{align*}
\Sigma_0: = Z(A,B)Z(C,D),
\end{align*}
\begin{align*}
\Sigma_1: =
& \sum_{\substack{W\subseteq C \\ Y \subseteq D\\ |W|=|Y| >0}} (WY)^{N} Z( C\smallsetminus W \cup \tfrac{1}{Y}, D\smallsetminus Y \cup \tfrac{1}{W}) \sum_{a\in A} \sum_{b\in B} \prod_{\substack{\hat{a}\in A \\ \hat{a}\neq a}} \left(1-\frac{\hat{a}}{a}\right)^{-1} \prod_{\substack{\hat{b}\in B \\ \hat{b}\neq b}} \left(1-\frac{\hat{b}}{b}\right)^{-1}
 \\
& \hspace{.8in}  \times   \sum_{\substack{Q\subseteq A\smallsetminus \{a\} \\ R \subseteq B\smallsetminus \{b\} \\ |Q|=|R| }}  \left( 1- \frac{ab(QR)^1}{(WY)^1} \right)^{-1} Z(A\smallsetminus \{a\}\smallsetminus Q \cup \tfrac{1}{R} , B\smallsetminus \{b\}\smallsetminus R \cup \tfrac{1}{Q}),
\end{align*}
and
\begin{align*}
\Sigma_2: = 
& \sum_{\substack{W\subseteq C \\ Y \subseteq D\\ |W|=|Y| }}   Z( C\smallsetminus W \cup \tfrac{1}{Y}, D\smallsetminus Y \cup \tfrac{1}{W}) \sum_{a\in A} \sum_{b\in B} \prod_{\substack{\hat{a}\in A \\ \hat{a}\neq a}} \left(1-\frac{\hat{a}}{a}\right)^{-1} \prod_{\substack{\hat{b}\in B \\ \hat{b}\neq b}} \left(1-\frac{\hat{b}}{b}\right)^{-1}
 \\
&  \hspace{.2in}  \times   \sum_{\substack{Q\subseteq A\smallsetminus \{a\} \\ R \subseteq B\smallsetminus \{b\} \\ |Q|=|R| }}  ( ab)^N(QR )^{N}\left(  1- \frac{(WY)^1} {ab(QR)^1} \right)^{-1} Z(A\smallsetminus \{a\}\smallsetminus Q \cup \tfrac{1}{R} , B\smallsetminus \{b\}\smallsetminus R \cup \tfrac{1}{Q}).
\end{align*}
More generally, \eqref{eqn: unitaryidentity0} holds for any finite multisets $A$ and $B$ of complex numbers, without any restrictions, provided that the right-hand side of \eqref{eqn: unitaryidentity0} is interpreted as its analytic continuation.
\end{theorem}

Theorem~\ref{thm: unitaryidentity} is directly analogous to \eqref{eqn: unitaryidentityzeta}, which is the identity that plays a key role in the heuristic in \cite{ck5}, because both Theorem~\ref{thm: unitaryidentity} and \eqref{eqn: unitaryidentityzeta} are the main tools used to carry out the summation over the ``twist parameters'' $\lambda$ or $M_1,\dots,M_{\ell},N_1,\dots,N_{\ell}$ in their respective settings.

We now outline the rest of the paper. In Section~\ref{sec: schur}, we state a few identities involving sums of products of Schur polynomials that we use in our proofs and might also be of independent interest. We prove these identities in Section~\ref{sec: schurproofs}. We prove the BK Splitting (Theorem~\ref{thm: BK}) in Section~\ref{sec: BK}. We prove Theorem~\ref{thm: twistedrecipe} in Section~\ref{sec: twistedrecipe} and Theorem~\ref{thm: unitaryidentity} in Section~\ref{sec: unitaryidentity}. In Section~\ref{sec: recipe}, we apply the BK Splitting and Theorem~\ref{thm: twistedrecipe} to prove \eqref{eqn: recipe} using lower twisted moments. In Appendix~\ref{sec: appendix}, we include for completeness proofs of certain known results that we use as lemmas in our proofs.

\section{Identities for sums of products of two Schur polynomials}\label{sec: schur}

In this section, we state lemmas involving Schur polynomials that may be of independent interest. We will prove these lemmas in Section~\ref{sec: schurproofs}.

\begin{lemma}\label{lem: partialcauchysumequalsmoment}
Let $A$ and $B$ be finite multisets of complex numbers. If $N$ is a positive integer, then
$$
\sum_{\substack{\nu \\ \nu_1 \leq N}} s_{\nu}(A)s_{\nu}(B) = \mathcal{M}_N(A,B;0),
$$
where the sum is over all partitions $\nu$ whose largest part $\nu_1$ is $\leq N$, $s_{\nu}(A)$ and $s_{\nu}(B)$ are defined by \eqref{eqn: schurpolynomialnotation}, and $\mathcal{M}_N(A,B;0)$ is defined by \eqref{eqn: twistedmomentdefinition} with $\lambda=0$, the zero partition.
\end{lemma}

\begin{lemma}\label{lem: partialcauchyrecipe}
Let $A$ and $B$ be finite sets of nonzero complex numbers such that $A$, $B$, $1/A$, and $1/B$ are pairwise disjoint. If $m$ is a nonnegative integer, then
\begin{equation}\label{eqn: partialcauchyrecipe0}
\sum_{\substack{\nu \\ \nu_1 \leq m}} s_{\nu}(A)s_{\nu}(B) = \sum_{\substack{U\subseteq A \\  V\subseteq B \\ |U|=|V|}} (UV)^m Z(A\smallsetminus U \cup \tfrac{1}{V}, B\smallsetminus V \cup \tfrac{1}{U}),
\end{equation}
where the sum over $\nu$ is over all partitions whose largest part $\nu_1$ is $\leq m$, $s_{\nu}(A)$ and $s_{\nu}(B)$ are defined by \eqref{eqn: schurpolynomialnotation}, and we are using the notations defined in \eqref{eqn: Zdef}, \eqref{eqn: UVNdef}, and \eqref{eqn: 1/Udef}. More generally, \eqref{eqn: partialcauchyrecipe0} holds for any finite multisets $A$ and $B$ of complex numbers, without any restrictions, provided that the right-hand side of \eqref{eqn: partialcauchyrecipe0} is interpreted as its analytic continuation.
\end{lemma}

Note that we allow $A$ or $B$ to be empty in Lemma~\ref{lem: partialcauchyrecipe}, in which case the equation still holds because both sides equal $1$ (recall that $s_{\rho}(\emptyset)=0$ if $\rho$ is a nonzero partition and $s_0(\emptyset)=1$; see the note below \eqref{eqn: schurpolynomialnotation}).

\begin{lemma}\label{lem: partialcauchywithextrafactorrecipe}
Let $A$ and $B$ be nonempty finite sets of nonzero complex numbers such that $A$, $B$, $1/A$, and $1/B$ are pairwise disjoint. Let $x$ be a complex number. If $n$ is a nonnegative integer, then
\begin{align}
\sum_{\substack{\nu \\ \nu_1 \leq n}} x^{\nu_1} s_{\nu}(A) s_{\nu}(B) =
& \sum_{a\in A} \sum_{b\in B} \prod_{\substack{\hat{a}\in A \\ \hat{a}\neq a}} \left(1-\frac{\hat{a}}{a}\right)^{-1} \prod_{\substack{\hat{b}\in B \\ \hat{b}\neq b}} \left(1-\frac{\hat{b}}{b}\right)^{-1} \notag
 \\
&  \times \sum_{0\leq m\leq n} (xab)^{m} \sum_{\substack{U\subseteq A\smallsetminus \{a\} \\ V\subseteq B\smallsetminus \{b\} \\ |U|=|V|}} (UV)^m Z(A\smallsetminus \{a\}\smallsetminus U \cup \tfrac{1}{V}, B\smallsetminus \{b\}\smallsetminus V \cup \tfrac{1}{U}), \label{eqn: partialcauchywithextrafactorrecipe}
\end{align}
where the sum over $\nu$ is over all partitions whose largest part $\nu_1$ is $\leq n$, $s_{\nu}(A)$ and $s_{\nu}(B)$ are defined by \eqref{eqn: schurpolynomialnotation}, and we are using the notations defined in \eqref{eqn: Zdef}, \eqref{eqn: UVNdef}, and \eqref{eqn: 1/Udef}. More generally, \eqref{eqn: partialcauchywithextrafactorrecipe} holds for any nonempty finite multisets $A$ and $B$ of complex numbers provided that the right-hand side of \eqref{eqn: partialcauchywithextrafactorrecipe} is interpreted as its analytic continuation.
\end{lemma}

\section{Proof of the BK Splitting}\label{sec: BK}

\begin{lemma}\label{lem: fourierexpansion}
Suppose that $f\in \mathbb{C}[x_1,\dots,x_n,x_1^{-1},\dots,x_n^{-1}]$ is symmetric, i.e., $f(x_1,\dots,x_n) = f(x_{\sigma(1)},\dots,x_{\sigma(n)})$ for every permutation $\sigma$ of $\{1,2,\dots,n\}$. Then there exist a finite set $E$ of dominant weights of length $n$ and a set $\{c_{\lambda} : \lambda \in E\}$ of complex numbers such that
$$
f(x_1,\dots,x_n) = \sum_{\lambda\in E} c_{\lambda} s_{\lambda}(x_1,\dots,x_n).
$$
In other words, the set of generalized Schur functions $s_{\lambda}(x_1,\dots,x_n)$, where $\lambda$ runs through the dominant weights of length $n$, is a spanning set of $\mathbb{C}[x_1,\dots,x_n,x_1^{-1},\dots,x_n^{-1}]$.
\end{lemma}
\begin{proof}
We first clear denominators and let $M$ be an integer such that
\begin{equation*}
(x_1\cdots x_n)^M f(x_1,\dots,x_n) 
\end{equation*}
is a polynomial in $\mathbb{C}[x_1,\dots,x_n]$. Then this product is a symmetric polynomial since $f$ is symmetric. Thus, by the well-known fact that the Schur polynomials form a linear basis for the set of symmetric polynomials (see, for example, (3.2) in Chapter~I of \cite{macdonald}; the proof there also holds true with $\mathbb{C}$ in place of $\mathbb{Z}$), we deduce that there is a finite set $D$ of partitions and a set $\{c_{\lambda} : \lambda \in D\}$ of complex numbers such that
\begin{equation*}
(x_1\cdots x_n)^M f(x_1,\dots,x_n) = \sum_{\lambda\in D} c_{\lambda} s_{\lambda}(x_1,\dots,x_n).
\end{equation*}
Divide both sides by $(x_1\cdots x_n)^M$ to deduce that
\begin{equation}\label{eqn: fourierexpansion1}
f(x_1,\dots,x_n) = \sum_{\lambda\in D} c_{\lambda} (x_1\cdots x_n)^{-M}s_{\lambda}(x_1,\dots,x_n).
\end{equation}
Now Jacobi's bialternant formula for Schur polynomials and the definition \eqref{eqn: genschurdef} of generalized Schur functions implies that if $\lambda=(\lambda_1,\dots,\lambda_n)$, then
\begin{align*}
(x_1\cdots x_n)^{-M} s_{\lambda}(x_1,\dots,x_n)
& = \frac{(x_1\cdots x_n)^{-M}\det\Big(x_j^{n-k+\lambda_k} \Big)_{1\leq j,k\leq n} }{ \det\Big(x_j^{n-k} \Big)_{1\leq j,k\leq n} } \\
& = \frac{\det\Big(x_j^{n-k+\lambda_k-M} \Big)_{1\leq j,k\leq n} }{ \det\Big(x_j^{n-k} \Big)_{1\leq j,k\leq n} } \\
& = s_{\lambda-\langle M^n \rangle }(x_1,\dots,x_n) ,
\end{align*}
where $\lambda-\langle M^n \rangle $ denotes the dominant weight $(\lambda_1-M,\dots,\lambda_n-M)$. The lemma follows from this and \eqref{eqn: fourierexpansion1}.
\end{proof}

The following lemma is equation (17) of Bump and Gamburd~\cite{bumpgamburd}. It is a generalization of the well-known fact that Schur polynomials are orthogonal over $U(N)$ (see, for example, Theorem~2.5 of \cite{bump} and Proposition~2 of \cite{bumpgamburd}). For completeness, we include a proof of this lemma in Appendix~\ref{sec: appendix}.

\begin{lemma}\label{lem: schurorthogonal}
For dominant weights $\lambda$ and $\rho$ of length $N$, let
\begin{equation*}
\mathcal{J}_N(\lambda,\rho) := \int_{U(N)} s_{\lambda}(t_1,\dots,t_N) \overline{s_{\rho}(t_1,\dots,t_N)}\,dg,
\end{equation*}
where the integral is over the unitary group $U(N)$ equipped with the Haar measure $dg$ such that $\int_{U(N)}1\,dg=1$ and the numbers $t_1,\dots,t_N$ are the eigenvalues of $g$, each repeated according to its multiplicity. Then
\begin{equation*}
\mathcal{J}_N(\lambda,\rho) = \begin{cases} 1 & \text{if } \lambda=\rho \\ 0 & \text{else.} \end{cases}
\end{equation*}
\end{lemma}

\begin{proof}
See Appendix~\ref{sec: appendix}.
\end{proof}

We now prove the BK Splitting (Theorem~\ref{thm: BK}). By factoring the characteristic polynomials completely, we see for $g\in U(N)$ that
\begin{equation*}
\prod_{\alpha \in A_2} \det(I-\alpha g^*) \prod_{\beta \in B_2} \det(I-\beta g) = \prod_{\alpha\in A_2}\prod_{j=1}^N (1-\alpha t_j^{-1}) \prod_{\beta\in B_2}\prod_{j=1}^N (1-\beta t_j),
\end{equation*}
where $t_1,t_2,\dots,t_N$ are the eigenvalues of $g$, each repeated according to its multiplicity. From this and Lemma~\ref{lem: fourierexpansion}, we deduce that
\begin{equation}\label{eqn: applyfourier}
\prod_{\alpha \in A_2} \det(I-\alpha g^*) \prod_{\beta \in B_2} \det(I-\beta g) = \sum_{\lambda} c_{\lambda} s_{\lambda} (t_1, \dots, t_N)
\end{equation}
for some complex numbers $c_{\lambda}$, where the sum is over all dominant weights of length $N$ and $c_{\lambda}\neq 0$ for at most finitely many $\lambda$. To determine the value of each $c_{\lambda}$, let $\rho$ be a dominant weight of length $N$, multiply both sides of \eqref{eqn: applyfourier} by $s_{-\rho}(t_1, \dots, t_N)$ (recall that if $\rho=(\rho_1,\rho_2,\dots,\rho_N)$, then $-\rho$ is defined to be the dominant weight $(-\rho_N,-\rho_{N-1},\dots,-\rho_1)$) and then integrate both sides of the resulting equation over $U(N)$ to deduce that
\begin{equation}\label{eqn: applyorthogonality}
\mathcal{M}_N(A_2,B_2;-\rho) = \sum_{\lambda} c_{\lambda}\int_{U(N)}  s_{\lambda} (t_1, \dots, t_N) s_{-\rho} (t_1, \dots, t_N) \,dg .
\end{equation}
Now the definition \eqref{eqn: genschurdef} and properties of determinants imply that
\begin{equation*}
s_{-\rho} (t_1, \dots, t_N) = s_{\rho}\big(t_1^{-1},\dots,t_N^{-1}\big) = \overline{ s_{\rho} (t_1, \dots, t_N)}. 
\end{equation*}
This, \eqref{eqn: applyorthogonality}, and Lemma~\ref{lem: schurorthogonal} imply that
\begin{equation*}
c_{\rho}= \mathcal{M}_N(A_2,B_2;-\rho)
\end{equation*}
for all dominant weights $\rho$ of length $N$. It follows from this and \eqref{eqn: applyfourier} that
\begin{equation*}
\prod_{\alpha \in A_2} \det(I-\alpha g^*) \prod_{\beta \in B_2} \det(I-\beta g) = \sum_{\lambda} \mathcal{M}_N(A_2,B_2;-\lambda) s_{\lambda} (t_1, \dots, t_N).
\end{equation*}
We insert this into the definition \eqref{eqn: twistedmomentdefinition} of $\mathcal{M}_N(A ,B ;0)$ to deduce that
\begin{align*}
\mathcal{M}_N(A,B;0) 
& = \int_{U(N)} \prod_{\alpha \in A_1} \det(I-\alpha g^*) \prod_{\beta \in B_1} \det(I-\beta g) \sum_{\lambda} \mathcal{M}_N(A_2,B_2;-\lambda) s_{\lambda} (t_1, \dots, t_N) \,dg \\
& = \sum_{\lambda} \mathcal{M}_N(A_2,B_2;-\lambda) \int_{U(N)} \prod_{\alpha \in A_1} \det(I-\alpha g^*) \prod_{\beta \in B_1} \det(I-\beta g) s_{\lambda} (t_1, \dots, t_N) \,dg \\
& = \sum_{\lambda} \mathcal{M}_N(A_2,B_2;-\lambda) \mathcal{M}_N(A_1,B_1;\lambda).
\end{align*}
This completes the proof of Theorem~\ref{thm: BK}.

\section{Proof of Theorem~\ref{thm: twistedrecipe}}\label{sec: twistedrecipe}

The following lemma is Lemma~1 of Bump and Gamburd~\cite{bumpgamburd} and follows immediately from the cofactor expansion of the determinant.

\begin{lemma}\label{lem: cofactorexpansion}
Suppose that $\tau$, $\gamma$, and $\delta$ are partitions such that $\tau = (\tau_1,\tau_2,\dots,\tau_{L+K})$, $\gamma = (\tau_1,\tau_2,\dots,\tau_{L})$, and $\delta = (\tau_{L+1},\tau_{L+2},\dots,\tau_{L+K})$. Let $\Xi_{L,K}$ denote the set of permutations $\sigma$ of $\{1,2,\dots,L+K\}$ such that $\sigma(1)<\sigma(2)<\cdots<\sigma(L)$ and $\sigma(L+1)<\sigma(L+2)<\cdots<\sigma(L+K)$, and let $\gamma + \langle K^L\rangle$ denote the partition $(\tau_1+K,\tau_2+K,\dots,\tau_{L}+K)$. Then
\begin{align*}
s_{\tau}(x_1,\dots,x_{L+K}) =
& \sum_{\sigma \in \Xi_{L,K}} \prod_{\substack{1\leq\ell\leq L \\ 1\leq k\leq K}} \big(x_{\sigma(\ell)} -x_{\sigma(L+k)}\big)^{-1} \\
& \times s_{\gamma +\langle K^L\rangle } \big(x_{\sigma(1)},x_{\sigma(2)},\dots, x_{\sigma (L)}\big) s_{\delta} \big(x_{\sigma(L+1)},x_{\sigma(L+2)},\dots, x_{\sigma (L+k)}\big).
\end{align*}
\end{lemma}

The following lemma is due to Bump and Gamburd~\cite[p.~243]{bumpgamburd}. We present a different proof using the Littlewood-Richardson rule in Appendix~\ref{sec: appendix}.

\begin{lemma}\label{lem: bumpgamburdlemma}
Let $\langle N^L \rangle$ denote the partition $(N,N,\dots,N)$ with $L$ parts. If $\mu=(\mu_1,\mu_2,\dots)$ and $\nu =(\nu_1,\nu_2,\dots)$ are partitions, then the Littlewood-Richardson coefficient $c^{\langle N^L \rangle}_{\mu,\nu}$ equals $0$ unless each of $\mu$ and $\nu$ has length at most $L$ and
$$
\mu_j + \nu_{L+1-j}=N
$$
for all $j=1,2,\dots,L$, in which case $c^{\langle N^L \rangle}_{\mu,\nu}$ equals $1$.
\end{lemma}
\begin{proof}
See Appendix~\ref{sec: appendix}.
\end{proof}

We now prove Theorem~\ref{thm: twistedrecipe}. Denote
\begin{equation}\label{eqn: Aalpha}
A=\{-\alpha_1^{-1},\dots,-\alpha_L^{-1}\}
\end{equation}
and
\begin{equation}\label{eqn: Balpha}
B=\{-\alpha_{L+1},\dots,-\alpha_{L+K}\}.
\end{equation}
With these notations, the definition \eqref{eqn: twistedmomentdefinition} may be written as
\begin{equation}\label{eqn: twistedmoment1}
\mathcal{M}_N(A,B;\lambda) = \int_{U(N)} \prod_{\ell=1}^L \det(I+\alpha_{\ell}^{-1}g^{*}) \prod_{k=1}^K \det(I +\alpha_{L+k} g) s_{\lambda} (t_1, \dots, t_N) \,dg.
\end{equation}
Further denote $\lambda=(\lambda_1,\dots,\lambda_N)$, and let
\begin{equation}\label{eqn: Mdef}
M = \left\{\begin{array}{cl}  -\lambda_N & \text{if } \lambda_N < 0 \\ 0 & \text{if } \lambda_N\geq 0. \end{array}\right.
\end{equation}
Then the sequence
\begin{equation}\label{eqn: makeapartition}
(\lambda_1+M,\dots,\lambda_N+M) = \rho, \ \text{say},
\end{equation}
is a partition because it is a decreasing sequence and $\lambda_N+M\geq 0$. The properties of determinants together with \eqref{eqn: genschurdef} and \eqref{eqn: makeapartition} imply that for $g$ as in \eqref{eqn: twistedmoment1}, we have
\begin{equation*}
\det(I+\alpha_{\ell}^{-1}g^{*}) = \overline{\det g} \det(I+\alpha_{\ell} g) \alpha_{\ell}^{-N}
\end{equation*}
for each $\ell$,
\begin{equation*}
\Big(\overline{\det g}\Big)^L (t_1\cdots t_N)^{-M} = \Big(\overline{t_1\cdots t_N}\Big)^{L+M} = \overline{ s_{ \langle (L+M)^N\rangle } (t_1,\dots,t_N) } ,
\end{equation*}
where $\langle (L+M)^N\rangle$ denotes the partition $(L+M,L+M,\dots,L+M)$ with $N$ parts, and
\begin{equation*}
(t_1\cdots t_N)^M s_{\lambda}(t_1,\dots,t_N) = s_{\rho}(t_1,\dots,t_N).
\end{equation*}
From these and \eqref{eqn: twistedmoment1}, we arrive at
\begin{equation}\label{eqn: twistedmoment2}
\mathcal{M}_N(A,B;\lambda)  = \prod_{\ell=1}^L \alpha_{\ell}^{-N}\int_{U(N)}
 \prod_{k=1}^{K+L} \det(I+\alpha_k g) \overline{ s_{ \langle (L+M)^N\rangle } (t_1,\dots,t_N) } s_{\rho}(t_1,\dots,t_N )\,dg.
\end{equation}
Now the dual Cauchy identity states that (see, for example, equation (4.3$'$) in Chapter~I of \cite{macdonald})
\begin{equation}\label{eqn: dualcauchyidentity}
\sum_{\mu} s_{\mu} (x_1,x_2,\dots,x_m) s_{\mu'}(y_1,y_2,\dots,y_n) = \prod_{k=1}^m \prod_{j=1}^n (1+x_k y_j),
\end{equation}
where the sum on the left-hand side is over all partitions $\mu$ (note that this sum has only finitely many nonzero terms by the fact that the length of $\mu'$ equals the largest part of $\mu$ and the convention that $s_{\nu}(x_1,\dots,x_m)=0$ if $\nu$ is a partition of length larger than $m$). We completely factor the characteristic polynomials in \eqref{eqn: twistedmoment2} and then apply the dual Cauchy identity \eqref{eqn: dualcauchyidentity} to deduce that
\begin{equation}\label{eqn: applydualcauchy}
\prod_{k=1}^{K+L} \det(I+\alpha_k g) = \prod_{k=1}^{K+L} \prod_{j=1}^N (1+\alpha_k t_j) = \sum_{\mu} s_{\mu} (\alpha_1,\dots,\alpha_{K+L}) s_{\mu'}(t_1,\dots,t_N).
\end{equation}
From this and \eqref{eqn: twistedmoment2}, we arrive at
\begin{align*}
\mathcal{M}_N(A,B;\lambda) = \prod_{\ell=1}^L \alpha_{\ell}^{-N}\int_{U(N)} \sum_{\mu} s_{\mu}(\alpha_1,\dots,\alpha_{K+L}) s_{\mu'}(t_1,\dots,t_N) \\
\times s_{\rho}(t_1,\dots,t_N ) \overline{s_{\langle (L+M)^N \rangle}(t_1,\dots,t_N)} \,dg.
\end{align*}
It follows from this and the definition of the Littlewood-Richardson coefficients (see, for example, equation (5.2) and \S 9 of Chapter~I of \cite{macdonald}) that
\begin{align*}
\mathcal{M}_N(A,B;\lambda)  =
& \prod_{\ell=1}^L \alpha_{\ell}^{-N}\sum_{\mu} s_{\mu}(\alpha_1,\dots,\alpha_{K+L})  \\
& \times \sum_{\nu} c^{\nu}_{\mu',\rho} \int_{U(N)} s_{\nu}(t_1,\dots,t_N)\overline{s_{\langle (L+M)^N \rangle}(t_1,\dots,t_N)}\,dg,
\end{align*}
where the $\nu$-sum is over all partitions $\nu$. We apply Lemma~\ref{lem: schurorthogonal} (or its well-known version in the special case when the generalized Schur functions are also Schur polynomials) to deduce that
\begin{equation*}
\mathcal{M}_N(A,B;\lambda)  = \prod_{\ell=1}^L \alpha_{\ell}^{-N}\sum_{\mu} s_{\mu}(\alpha_1,\dots,\alpha_{K+L})   c^{\langle (L+M)^N \rangle}_{\mu',\rho}.
\end{equation*}
This is equivalent to
\begin{equation}\label{eqn: twistedmoment3}
\mathcal{M}_N(A,B;\lambda)  = \prod_{\ell=1}^L \alpha_{\ell}^{-N}\sum_{\mu} s_{\mu}(\alpha_1,\dots,\alpha_{K+L})   c^{\langle N^{L+M} \rangle}_{\mu,\rho'}
\end{equation}
because, in general, $c^{\lambda}_{\mu,\nu} = c^{\lambda'}_{\mu',\nu'}$ (see, for example, (2.7) and (3.8) of Chapter I of MacDonald). Now write $\rho'$ as
\begin{equation}\label{eqn: rhoprime}
\rho' = (h_1,h_2,\dots).
\end{equation}
From \eqref{eqn: twistedmoment3} and Lemma~\ref{lem: bumpgamburdlemma}, we deduce that if $\rho'$ has length $>L+M$, then
\begin{equation}\label{eqn: twistedmoment4a}
\mathcal{M}_N(A,B;\lambda)  = 0,
\end{equation}
while if $\rho'$ has length $\leq L+M$, then
\begin{equation}\label{eqn: twistedmoment4b}
\mathcal{M}_N(A,B;\lambda)  = s_{\tau}(\alpha_1,\dots,\alpha_{K+L}) \prod_{\ell=1}^L \alpha_{\ell}^{-N} ,
\end{equation}
where $\tau$ is the partition $\tau = (N-h_{L+M},N-h_{L+M-1},\dots,N-h_1)$. Recalling the definition \eqref{eqn: makeapartition} of $\rho$, we see that the length of its conjugate partition $\rho'$ is exactly $\lambda_1+M$. Thus, \eqref{eqn: twistedmoment4a} holds if $\lambda_1>L$, while \eqref{eqn: twistedmoment4b} holds if $\lambda_1\leq L$. Moreover, if $-\lambda_N>K$, then $\lambda_N+M=0$ by the definition \eqref{eqn: Mdef} of $M$, which implies that \eqref{eqn: makeapartition} has length $<N$ and thus $h_1<N$ in \eqref{eqn: rhoprime}. In this case, the partition $\tau$ in \eqref{eqn: twistedmoment4b} has length exactly $L+M$, so that $s_{\tau}(\alpha_1,\dots,\alpha_{K+L})=0$ because $L+M>L+K$ by the assumption $-\lambda_N>K$ and the definition \eqref{eqn: Mdef} of $M$. In conclusion, \eqref{eqn: twistedmoment4a} holds if $\lambda_1>L$ or $-\lambda_N>K$, and \eqref{eqn: twistedmoment4b} holds if $\lambda_1\leq L$ and $-\lambda_N\leq K$. This proves Theorem~\ref{thm: twistedrecipe} when $\lambda_1>L$ or $-\lambda_N>K$ because, in this case, if $\mu$ and $\nu$ are as in the statement of Theorem~\ref{thm: twistedrecipe}, then either the length $\max\{\lambda_1,0\}$ of $\mu'$ is $>L$ or the length $\max\{-\lambda_N,0\}$ of $\nu'$ is $>K$, which by \eqref{eqn: Aalpha} and \eqref{eqn: Balpha} means that all Schur polynomials in the conclusion of Theorem~\ref{thm: twistedrecipe} are zero by the standard convention that a Schur polynomial $s_{\nu}(x_1,\dots,x_n)$ is zero unless the length of the partition $\nu$ is $\leq n$.

To complete the proof of Theorem~\ref{thm: twistedrecipe}, we now assume that $\lambda_1\leq L$ and $-\lambda_N\leq K$, so that \eqref{eqn: twistedmoment4b} holds. Let $\mu$ and $\nu$ be as defined in the statement of Theorem~\ref{thm: twistedrecipe}. Then $\mu'$ has length $\max\{\lambda_1,0\}$, which is $\leq L$ by our current assumption $\lambda_1\leq L$. Thus, we may denote
\begin{equation}\label{eqn: muprime}
\mu' = (\mu_1',\mu_2',\dots,\mu_L').
\end{equation}
Similarly, $\nu'$ has length $M$ by \eqref{eqn: Mdef} and we may thus denote
\begin{equation}\label{eqn: nuprime}
\nu' = (\nu_1',\nu_2',\dots,\nu_M').
\end{equation}
The definitions of $\mu$ and $\nu$ in the statement of Theorem~\ref{thm: twistedrecipe} and the definitions \eqref{eqn: Mdef}, \eqref{eqn: makeapartition}, \eqref{eqn: rhoprime}, \eqref{eqn: muprime}, and \eqref{eqn: nuprime} imply that 
\begin{equation*}
(h_1,\dots,h_M) = (N-\nu_M', N-\nu_{M-1}',\dots,N-\nu_1')
\end{equation*}
and
\begin{equation*}
(h_{M+1}, \dots, h_{M+L}) = (\mu_1',\dots,\mu_L').
\end{equation*}
It follows that the partition $\tau$ defined in \eqref{eqn: twistedmoment4b} equals
\begin{equation*}
\tau  = (N-\mu_L',N-\mu'_{L-1},\dots,N-\mu'_1, \nu_1',\nu_2',\dots,\nu_M').
\end{equation*}
From this, \eqref{eqn: twistedmoment4b}, and Lemma~\ref{lem: cofactorexpansion} with $\gamma = (N-\mu_L',N-\mu'_{L-1},\dots,N-\mu'_1)$ and $\delta = \nu'$, we deduce that
\begin{align}
\mathcal{M}_N(A,B;\lambda) = & \sum_{\sigma \in \Xi_{L,K}} \prod_{\substack{1\leq \ell\leq L \\ 1\leq k\leq K }} \big(\alpha_{\sigma(\ell)} -\alpha_{\sigma(L+k)}\big)^{-1} \prod_{\ell=1}^L \alpha_{\ell}^{-N} \notag\\
& \times  s_{ \gamma+ \langle K^L\rangle }  \big(\alpha_{\sigma(1)},\dots,\alpha_{\sigma(L)}\big) s_{\nu'} \big(\alpha_{\sigma(L+1)},\dots, \alpha_{\sigma(L+K)}\big). \label{eqn: twistedmoment5}
\end{align}
To simplify the right-hand side, we use the properties of determinants and Jacobi's bialternant formula to deduce that
\begin{align*}
s_{ \gamma+ \langle K^L\rangle } (x_1,\dots,x_L) & =  (x_1\cdots x_L)^{K+N} \frac{\det\left(x_j^{i-1-\mu_{i}'}\right)_{1\leq i,j\leq L} }{\det\left(x_j^{i-1}\right)_{1\leq i,j\leq L}}\\
& =  (x_1\cdots x_L)^{K+N}\frac{\det\left(\left(\frac{1}{x_j}\right)^{L-i+\mu_{i}'}\right)_{1\leq i,j\leq L} }{\det\left(\left(\frac{1}{x_j}\right)^{L-i}\right)_{1\leq i,j\leq L}} \\
& = (x_1\cdots x_L)^{K+N}s_{\mu'} \left( \frac{1}{x_1},\dots,\frac{1}{x_L}\right)
\end{align*}
by \eqref{eqn: muprime}. From this and \eqref{eqn: twistedmoment5}, we arrive at
\begin{align}
\mathcal{M}_N(A,B;\lambda) =
& \sum_{\sigma \in \Xi_{L,K}} \prod_{\substack{1\leq \ell\leq L \\ 1\leq k\leq K }} \big(1-\alpha_{\sigma(\ell)}^{-1} \alpha_{\sigma(L+k)}\big)^{-1} \prod_{k=1}^K\big(\alpha_{\sigma(L+k)}^{-1} \alpha_{L+k}\big)^{N}  \notag\\
& \times  s_{\mu'} \left( \frac{1}{\alpha_{\sigma(1)}},\dots,\frac{1}{\alpha_{\sigma(L)}}\right)s_{\nu'} \big(\alpha_{\sigma(L+1)},\dots, \alpha_{\sigma(L+K)}\big). \label{eqn: twistedmoment6}
\end{align}
Recalling \eqref{eqn: Aalpha}, \eqref{eqn: Balpha}, and the notations described in \eqref{eqn: Zdef}, \eqref{eqn: UVNdef}, and \eqref{eqn: 1/Udef}, we make a change of variables in \eqref{eqn: twistedmoment6} by setting
\begin{equation}\label{eqn: Uchangevar}
U=\big\{-\alpha_{\sigma(L+k)}^{-1}: 1\leq k\leq K \text{ and } \sigma(L+k) \leq L\big\}   
\end{equation}
and
\begin{equation}\label{eqn: Vchangevar}
V = \big\{-\alpha_{L+k}: 1\leq k\leq K \text{ and } L+k \neq \sigma(L+j) \text{ for all } 1\leq j\leq K \big\}
\end{equation}
to write \eqref{eqn: twistedmoment6} as
\begin{align*}
\mathcal{M}_N(A,B;\lambda) = \sum_{\substack{U\subseteq A \\ V\subseteq B \\ |U|=|V|}}
& Z\big(A\smallsetminus U \cup \tfrac{1}{V}, B\smallsetminus V \cup \tfrac{1}{U}\big) (UV)^N   \notag\\
& \times  s_{\mu'}\big(-(A\smallsetminus U \cup \tfrac{1}{V})\big) s_{\nu'} \big(-(B\smallsetminus V \cup \tfrac{1}{U})\big).
\end{align*}
This completes the proof of Theorem~\ref{thm: twistedrecipe}.

\section{Proofs of the identities for sums of products of two Schur polynomials}\label{sec: schurproofs}

\begin{proof}[Proof of Lemma \ref{lem: partialcauchysumequalsmoment}]
Using an argument similar to that in \eqref{eqn: applydualcauchy}, we see from the definition \eqref{eqn: twistedmomentdefinition} and the dual Cauchy identity \eqref{eqn: dualcauchyidentity} that
\begin{equation*}
\mathcal{M}_N(A,B;0) =\int_{U(N)}  \sum_{\mu} s_{\mu}(-A) \overline{s_{\mu'}(T)}  \sum_{\nu} s_{\nu}(-B)s_{\nu'}(T)\,dg,
\end{equation*}
where $T=\{t_1,\dots,t_N\}$ is the multiset of eigenvalues of $g$, with each eigenvalue repeated according to its multiplicity, and the sums over $\mu$ and $\nu$ each have only finitely many nonzero terms. The factor $s_{\nu'}(T)$ is zero for all $g\in U(N)$ unless the largest part $\nu_1$ of $\nu$ is $\leq N$. Similarly, $\overline{s_{\mu'}(T)}=0$ for all $g\in U(N)$ unless the largest part $\mu_1$ of $\mu$ is $\leq N$. From these and Lemma~\ref{lem: schurorthogonal} (or its well-known version in the special case when the generalized Schur functions are also Schur polynomials), we deduce that
\begin{equation*}
\mathcal{M}_N(A,B;0) = \sum_{\substack{\nu \\ \nu_1 \leq N}} s_{\nu}(-A)s_{\nu}(-B).
\end{equation*}
Lemma~\ref{lem: partialcauchysumequalsmoment} follows from this and the fact that the Schur polynomial $s_{\nu}$ is a homogeneous polynomial of degree $|\nu|$.
\end{proof}

\begin{proof}[Proof of Lemma \ref{lem: partialcauchyrecipe}]
Lemma \ref{lem: partialcauchyrecipe} with $m\geq 1$ follows immediately from Lemma~\ref{lem: partialcauchysumequalsmoment} and \eqref{eqn: recipe} (or Theorem~\ref{thm: twistedrecipe} with $\lambda=0$). It is left to prove Lemma~\ref{lem: partialcauchyrecipe} with $m=0$. Observe that
\begin{equation}\label{eqn: partialcauchyrecipe1}
\sum_{\substack{\nu \\ \nu_1 \leq 0}} s_{\nu}(A)s_{\nu}(B) = s_0(A)s_0(B) = 1.
\end{equation}
On the other hand, if we label the elements in $A\cup B$ as in \eqref{eqn: Aalpha} and \eqref{eqn: Balpha}, then we may make a change of variables $U,V\mapsto \sigma$ defined by \eqref{eqn: Uchangevar} and \eqref{eqn: Vchangevar} to write
\begin{equation}\label{eqn: partialcauchyrecipe2}
\sum_{\substack{U\subseteq A \\ V\subseteq B \\ |U|=|V|}}  Z(A\smallsetminus U \cup \tfrac{1}{V}, B\smallsetminus V \cup \tfrac{1}{U}) = \sum_{\sigma \in \Xi_{L,K}}  \prod_{\substack{1\leq \ell\leq L \\ 1\leq k\leq K }} \big(1-\alpha_{\sigma(\ell)}^{-1} \alpha_{\sigma(L+k)}\big)^{-1},
\end{equation}
where $\Xi_{L,K}$ is as defined in Lemma~\ref{lem: cofactorexpansion}. Now if $\langle K^L\rangle$ is the partition $(K,K,\dots,K)$ with $K$ repeated $L$ times, then
\begin{equation*}
s_{\langle K^L \rangle } (x_1,\dots,x_L) = (x_1\cdots x_L)^K.
\end{equation*}
From this, \eqref{eqn: partialcauchyrecipe2}, and Lemma~\ref{lem: cofactorexpansion} with $\tau=\gamma=\delta=0$, we deduce that
\begin{align*}
\sum_{\substack{U\subseteq A \\ V\subseteq B \\ |U|=|V|}}  Z(A\smallsetminus U \cup \tfrac{1}{V}, B\smallsetminus V \cup \tfrac{1}{U})
& = \sum_{\sigma \in \Xi_{L,K}}  s_{\langle K^L \rangle } \big(\alpha_{\sigma(1)},\dots,\alpha_{\sigma(L)}\big) \prod_{\substack{1\leq \ell\leq L \\ 1\leq k\leq K }} \big(\alpha_{\sigma(\ell)}- \alpha_{\sigma(L+k)}\big)^{-1} \\
& = s_0(\alpha_1,\dots,\alpha_{L+K}) = 1.
\end{align*}
Lemma~\ref{lem: partialcauchyrecipe} with $m=0$ follows from this and \eqref{eqn: partialcauchyrecipe1}, and the proof of Lemma~\ref{lem: partialcauchyrecipe} is complete.
\end{proof}

\begin{proof}[Proof of Lemma \ref{lem: partialcauchywithextrafactorrecipe}]

As $\nu$ runs through all partitions $(\nu_1,\nu_2,\dots)$, the number $\nu_1$ runs through all nonnegative integers and the partition $(\nu_2,\nu_3,\dots)$ runs through all partitions whose largest part $\nu_2$ is $\leq \nu_1$. Since $A$ is nonempty, Lemma~\ref{lem: cofactorexpansion} with $\tau = \nu=(\nu_1,\nu_2,\dots)$, $\gamma=(\nu_1)$, and $\delta = (\nu_2,\nu_3,\dots)$ gives
\begin{equation*}
s_{\nu}(A)  = \sum_{a\in A}\prod_{\substack{\hat{a}\in A \\ \hat{a}\neq a}}(a-\hat{a})^{-1} a^{\nu_1+|A|-1} s_{\delta}(A\smallsetminus \{a\}).
\end{equation*}
Similarly,
\begin{equation*}
s_{\nu}(B) = \sum_{b\in B}\prod_{\substack{\hat{b}\in B \\ \hat{b}\neq b}}(b-\hat{b})^{-1} b^{\nu_1+|B|-1} s_{\delta}(B\smallsetminus \{b\}).
\end{equation*}
Therefore
\begin{align*}
& \sum_{\substack{\nu \\ \nu_1 \leq n}} x^{\nu_1} s_{\nu}(A) s_{\nu}(B) \\
& =  \sum_{a\in A}\sum_{b\in B}   \prod_{\substack{\hat{a}\in A \\ \hat{a}\neq a}} \left(1-\frac{\hat{a}}{a}\right)^{-1} \prod_{\substack{\hat{b}\in B \\ \hat{b}\neq b}} \left(1-\frac{\hat{b}}{b}\right)^{-1} \sum_{0\leq \nu_1\leq n} (xab)^{\nu_1} \sum_{\substack{\delta \\ \delta_1\leq \nu_1}}  s_{\delta}(A\smallsetminus \{a\})   s_{\delta}(B\smallsetminus \{b\}),
\end{align*}
where the summation variable $\delta$ runs over all partitions whose largest part $\delta_1$ is $\leq \nu_1$. Evaluating the $\delta$-sum via Lemma~\ref{lem: partialcauchyrecipe} completes the proof of Lemma~\ref{lem: partialcauchywithextrafactorrecipe}.
\end{proof}

\section{Proof of Theorem~\ref{thm: unitaryidentity}}\label{sec: unitaryidentity}

We now prove Theorem~\ref{thm: unitaryidentity}. We first evaluate the $\nu$-sum using Lemma~\ref{lem: partialcauchyrecipe} to deduce that
\begin{align*}
& \sum_{\substack{\mu,\nu \\ \mu_1+\nu_1\leq N}} s_{\mu}(A)s_{\mu}(B) s_{\nu}(C)s_{\nu}(D) \\
& =  \sum_{\substack{W\subseteq C\\  Y\subseteq D \\ |W|=|Y|}} (WY)^{N} Z( C\smallsetminus W \cup \tfrac{1}{Y}, D\smallsetminus Y \cup \tfrac{1}{W})\sum_{\substack{ \mu \\ \mu_1\leq N}} (WY)^{-\mu_1} s_{\mu}(A)s_{\mu}(B).
\end{align*}
We then evaluate the latter $\mu$-sum via Lemma~\ref{lem: partialcauchywithextrafactorrecipe} and arrive at
\begin{align*}
& \sum_{\substack{\mu,\nu \\ \mu_1+\nu_1\leq N}} s_{\mu}(A)s_{\mu}(B) s_{\nu}(C)s_{\nu}(D) \\
& = \sum_{\substack{W\subseteq C\\  Y\subseteq D \\ |W|=|Y|}} (WY)^{N} Z( C\smallsetminus W \cup \tfrac{1}{Y}, D\smallsetminus Y \cup \tfrac{1}{W}) \sum_{a\in A} \sum_{b\in B} \prod_{\substack{\hat{a}\in A \\ \hat{a}\neq a}} \left(1-\frac{\hat{a}}{a}\right)^{-1} \prod_{\substack{\hat{b}\in B \\ \hat{b}\neq b}} \left(1-\frac{\hat{b}}{b}\right)^{-1} \\
& \ \ \ \ \times \sum_{0\leq m\leq N} \frac{(ab)^m}{(WY)^m} \sum_{\substack{Q\subseteq A\smallsetminus \{a\} \\ R\subseteq B\smallsetminus \{b\} \\ |Q|=|R|}} (QR)^m Z(A\smallsetminus \{a\}\smallsetminus Q \cup \tfrac{1}{R}, B\smallsetminus \{b\}\smallsetminus R \cup \tfrac{1}{Q}).
\end{align*}
We may evaluate the $m$-sum by using the formula for a sum of a geometric series to deduce that
\begin{align*}
\sum_{0\leq m\leq N}  \frac{(ab)^m(QR)^m}{(WY)^m}  =  \left(1- \frac{ab(QR)^1}{(WY)^1}\right)^{-1} +  \frac{(ab)^N(QR)^N}{(WY)^N}  \left( 1- \frac{(WY)^1}{ab(QR)^1} \right)^{-1}
\end{align*}
and thus
\begin{equation*}
\sum_{\substack{\mu,\nu \\ \mu_1+\nu_1\leq N}} s_{\mu}(A)s_{\mu}(B) s_{\nu}(C)s_{\nu}(D) = \Upsilon  + \Sigma_2,
\end{equation*}
where
\begin{align}
\Upsilon :=
& \sum_{\substack{W\subseteq C\\  Y\subseteq D \\ |W|=|Y|}} (WY)^{N} Z( C\smallsetminus W \cup \tfrac{1}{Y}, D\smallsetminus Y \cup \tfrac{1}{W}) \sum_{a\in A} \sum_{b\in B} \prod_{\substack{\hat{a}\in A \\ \hat{a}\neq a}} \left(1-\frac{\hat{a}}{a}\right)^{-1} \prod_{\substack{\hat{b}\in B \\ \hat{b}\neq b}} \left(1-\frac{\hat{b}}{b}\right)^{-1} \notag\\
& \times  \sum_{\substack{Q\subseteq A\smallsetminus \{a\} \\ R\subseteq B\smallsetminus \{b\} \\ |Q|=|R|}} \left( 1- \frac{ab(QR)^1}{(WY)^1}\right)^{-1} Z(A\smallsetminus \{a\}\smallsetminus Q \cup \tfrac{1}{R}, B\smallsetminus \{b\}\smallsetminus R \cup \tfrac{1}{Q}) \label{eqn: upsilondef}
\end{align}
and $\Sigma_2$ is as defined in the statement of Theorem~\ref{thm: unitaryidentity}. Now the sum of the terms in \eqref{eqn: upsilondef} with $|W|>0$ is precisely the sum $\Sigma_1$ in the statement of Theorem~\ref{thm: unitaryidentity}. Thus, to complete the proof of Theorem~\ref{thm: unitaryidentity}, it is left to evaluate the sum of the terms in \eqref{eqn: upsilondef} with $|W|=|Y|=0$. Call this sum $\Upsilon_0$, so that
\begin{align}
\Upsilon_0 :=
& Z( C, D) \sum_{a\in A} \sum_{b\in B} \prod_{\substack{\hat{a}\in A \\ \hat{a}\neq a}} \left(1-\frac{\hat{a}}{a}\right)^{-1} \prod_{\substack{\hat{b}\in B \\ \hat{b}\neq b}} \left(1-\frac{\hat{b}}{b}\right)^{-1} \notag \\
& \times \sum_{\substack{Q\subseteq A\smallsetminus \{a\} \\ R\subseteq B\smallsetminus \{b\} \\ |Q|=|R|}} \big( 1-  ab(QR)^1 \big)^{-1} Z(A\smallsetminus \{a\}\smallsetminus Q \cup \tfrac{1}{R}, B\smallsetminus \{b\}\smallsetminus R \cup \tfrac{1}{Q}). \label{eqn: upsilon0def}
\end{align}
To evaluate $\Upsilon_0$, we first assume that the elements of $A\cup B$ each have absolute value $<1$. The general result will follow from meromorphic continuation. With this assumption, we may write $( 1-  ab(QR)^1 )^{-1}$ as the sum of an infinite geometric series and deduce from \eqref{eqn: upsilon0def} that
\begin{align*}
\Upsilon_0 = & Z( C, D) \sum_{a\in A} \sum_{b\in B} \prod_{\substack{\hat{a}\in A \\ \hat{a}\neq a}} \left(1-\frac{\hat{a}}{a}\right)^{-1} \prod_{\substack{\hat{b}\in B \\ \hat{b}\neq b}} \left(1-\frac{\hat{b}}{b}\right)^{-1} \\
& \times \sum_{m=0}^{\infty} (ab)^m \sum_{\substack{Q\subseteq A\smallsetminus \{a\} \\ R\subseteq B\smallsetminus \{b\} \\ |Q|=|R|}}  (QR)^m Z(A\smallsetminus \{a\}\smallsetminus Q \cup \tfrac{1}{R}, B\smallsetminus \{b\}\smallsetminus R \cup \tfrac{1}{Q}).
\end{align*}
From this and Lemma~\ref{lem: partialcauchyrecipe}, we arrive at
\begin{equation}\label{eqn: upsilon0}
\Upsilon_0 = Z( C, D) \sum_{a\in A} \sum_{b\in B} \prod_{\substack{\hat{a}\in A \\ \hat{a}\neq a}} \left(1-\frac{\hat{a}}{a}\right)^{-1} \prod_{\substack{\hat{b}\in B \\ \hat{b}\neq b}} \left(1-\frac{\hat{b}}{b}\right)^{-1} \sum_{m=0}^{\infty} (ab)^m \sum_{\substack{\nu \\ \nu_1 \leq m}} s_{\nu}(A\smallsetminus\{a\})s_{\nu}(B\smallsetminus\{b\}).
\end{equation}
Now if $m$ is a nonnegative integer, $\nu=(\nu_1,\nu_2,\dots)$ is a partition with $\nu_1\leq m$, and $\mu$ is the partition $(m,\nu_1,\nu_2,\dots)$, then Lemma~\ref{lem: cofactorexpansion} with $\tau=\mu$, $\gamma=(m)$, and $\delta=\nu$ implies
\begin{equation*}
\sum_{a\in A} \prod_{\substack{\hat{a}\in A \\ \hat{a}\neq a}} \left(1-\frac{\hat{a}}{a}\right)^{-1} a^m s_{\nu} (A\smallsetminus \{a\}) = s_{\mu}(A)
\end{equation*}
and
\begin{equation*}
\sum_{b\in B}  \prod_{\substack{\hat{b}\in B \\ \hat{b}\neq b}} \left(1-\frac{\hat{b}}{b}\right)^{-1} b^m s_{\nu} (B\smallsetminus \{b\})  = s_{\mu}(B).
\end{equation*}
From these and \eqref{eqn: upsilon0}, we deduce that
\begin{equation*}
\Upsilon_0 =  Z( C, D)\sum_{m=0}^{\infty} \sum_{\substack{\nu \\ \nu_1 \leq m}}   s_{\mu}(A) s_{\mu}(B),
\end{equation*}
where $\mu=\mu(m,\nu)$ is the partition $(m,\nu_1,\nu_2,\dots)$ when $\nu=(\nu_1,\nu_2\dots)$. This simplifies to 
\begin{equation*}
\Upsilon_0 =  Z( C, D) \sum_{\mu} s_{\mu}(A) s_{\mu}(B).
\end{equation*}
Since each element of $A\cup B$ has absolute value $<1$, we may evaluate the $\mu$-sum using Cauchy's identity (see, for example, equation (4.3) in Chapter~I of \cite{macdonald}, which holds true when all the variables are complex numbers of moduli $<1$) and arrive at
\begin{equation*}
\Upsilon_0 =  Z(A,B)Z(C,D),
\end{equation*}
which is $\Sigma_0$ in the statement of Theorem~\ref{thm: unitaryidentity}. We have thus proved that $\Upsilon_0=\Sigma_0$ under the additional assumption that the elements of $A\cup B$ each have absolute value $<1$. It follows from this and meromorphic continuation that $\Upsilon_0=\Sigma_0$ regardless of the sizes of the elements in $A\cup B$ because $\Upsilon_0$ and $\Sigma_0$ are both rational functions by \eqref{eqn: upsilon0def} and \eqref{eqn: Zdef}. This completes the proof of Theorem~\ref{thm: unitaryidentity}.

\section{Using lower twisted moments to evaluate higher untwisted moments}\label{sec: recipe}

In this section, we use the BK Splitting and Theorem~\ref{thm: twistedrecipe} for lower twisted moments to prove \eqref{eqn: recipe} for higher (untwisted) moments.

\begin{lemma}\label{lem: linearindependence}
Let $L,K$ be positive integers. For each pair $U,V$ of sets such that $U\subseteq\{1,2,\dots,L\}$ and $V\subseteq \{1,2,\dots,K\}$, let $f_{U,V}(x_1,\dots,x_L;y_1,\dots,y_{K})$ be a complex-valued function that is entire as a function of each of the variables $x_1,\dots,x_L,y_1,\dots,y_{K}$. Suppose there is a positive integer $N_0$ such that
\begin{equation}\label{eqn: linearindependence}
\sum_{\substack{U \subseteq  \{1,\dots,L\} \\ V\subseteq \{1,\dots,K\}}} \Bigg( \prod_{i\in U} x_i \prod_{j\in V} y_j \Bigg)^N f_{U,V}(x_1,\dots,x_L;y_1,\dots,y_{K}) = 0
\end{equation}
for all integers $N\geq N_0$ and all complex numbers $x_1,\dots,x_L,y_1,\dots,y_{K}$. Then $f_{U,V}$ is the zero function for all $U,V$.
\end{lemma}

\begin{proof}
Fix an arbitrary pair $U_0,V_0$ of sets such that $U_0\subseteq\{1,2,\dots,L\}$ and $V_0\subseteq \{1,2,\dots,K\}$. Let $\mathcal{R}$ be the subset of $\mathbb{C}^{L+K}$ defined by
\begin{align*}
\mathcal{R}:= \{ (x_1,\dots,x_L,y_1,\dots,y_{K})
& \in \mathbb{C}^{L+K} \ : \ 
|x_i| >1 \text{ for all } i\in U_0, \ |y_j|>1 \text{ for all } j\in V_0,\\
& 0<|x_i|<1 \text{ for all } i\not\in U_0, \text{ and } 0<|y_j|<1 \text{ for all } j\not\in V_0 \}.
\end{align*}
Suppose that $(x_1,\dots,x_L,y_1,\dots,y_{K})\in \mathcal{R}$ and $N\geq N_0$ is an integer, so that \eqref{eqn: linearindependence} holds. Divide both sides of \eqref{eqn: linearindependence} by $\prod_{i\in U_0} x_i^N \prod_{j\in V_0} y_j^N$ to deduce that
\begin{equation}\label{eqn: linearindependence2}
\sum_{\substack{U \subseteq  \{1,\dots,L\} \\ V\subseteq \{1,\dots,K\}}} \Bigg(\frac{ \prod_{i\in U\smallsetminus U_0} x_i \prod_{j\in V\smallsetminus V_0} y_j}{\prod_{i\in U_0\smallsetminus U} x_i  \prod_{j\in V_0\smallsetminus V} y_j } \Bigg)^N f_{U,V}(x_1,\dots,x_L;y_1,\dots,y_{K}) = 0.
\end{equation}
Since $(x_1,\dots,x_L,y_1,\dots,y_{K})\in \mathcal{R}$, if $U\subseteq\{1,2,\dots,L\}$ and $V\subseteq \{1,2,\dots,K\}$ such that $U\neq U_0$ or $V\neq V_0$, then
\begin{align*}
\Bigg| \frac{ \prod_{i\in U\smallsetminus U_0} x_i \prod_{j\in V\smallsetminus V_0} y_j}{\prod_{i\in U_0\smallsetminus U} x_i  \prod_{j\in V_0\smallsetminus V} y_j } \Bigg| <1.
\end{align*}
Thus, taking $N\rightarrow \infty$ in \eqref{eqn: linearindependence2}, we deduce that
\begin{equation}\label{eqn: linearindependence3}
f_{U_0,V_0}(x_1,\dots,x_L;y_1,\dots,y_{K}) = 0
\end{equation}
for all $(x_1,\dots,x_L,y_1,\dots,y_{K})\in \mathcal{R}$. It follows from this and analytic continuation that \eqref{eqn: linearindependence3} holds for all $(x_1,\dots,x_L,y_1,\dots,y_{K})\in \mathbb{C}^{L+K}$. Since $U_0$ and $V_0$ are arbitrary, this proves the lemma.
\end{proof}

We now prove \eqref{eqn: recipe} for $|A|\geq 2$ and $|B|\geq 2$ via the BK Splitting and Theorem~\ref{thm: twistedrecipe}. For conciseness and ease of notation, we assume that $A$ and $B$ are finite sets of nonzero complex numbers such that $A$, $B$, $1/A$, and $1/B$ are pairwise disjoint. The more general result with $A$ and $B$ finite multisets of complex numbers will follow immediately from analytic continuation. The proof itself may be adapted to the general case by changing all set operations to their analogues for multisets and by counting multiplicities when taking sums or products involving multisets.

In our proof, we will use Theorem~\ref{thm: twistedrecipe} only for lower twisted moments, i.e., for twisted moments $\mathcal{M}_N(C,D;\lambda)$ with $|C|<|A|$ and $|D|<|B|$. We will also apply Theorem~\ref{thm: unitaryidentity}, which uses \eqref{eqn: recipe} implicitly in its proof (see the proofs of Lemmas~\ref{lem: partialcauchyrecipe} and \ref{lem: partialcauchywithextrafactorrecipe} in Section~\ref{sec: schurproofs}). Therefore our proof will also implicitly use \eqref{eqn: recipe}, but only for lower moments, i.e., for moments $\mathcal{M}_N(C,D;0)$ with $|C|<|A|$ and $|D|<|B|$.

Let $A=A_1\cup A_2$ be a partition of $A$ into disjoint nonempty subsets, and similarly, let $B=B_1\cup B_2$ be a partition of $B$ into disjoint nonempty subsets (such partitions exist because $|A|\geq 2$ and $|B|\geq 2$). Then Theorem~\ref{thm: BK} gives
\begin{equation*}
\mathcal{M}_N(A,B;0) = \sum_{\lambda} \mathcal{M}_N(A_1,B_1;\lambda) \mathcal{M}_N(A_2,B_2;-\lambda),
\end{equation*}
where the sum on the right-hand side is over all dominant weights $\lambda$ of length $N$. We evaluate the lower twisted moments $\mathcal{M}_N(A_1,B_1;\lambda)$ and $\mathcal{M}_N(A_2,B_2;-\lambda)$ using Theorem~\ref{thm: twistedrecipe} and arrive at 
\begin{align*}
\mathcal{M}_N(A,B;0) =
& \sum_{\lambda} \sum_{ \substack{U_1\subseteq A_1 \\ V_1\subseteq B_1 \\ |U_1|=|V_1| } } (U_1V_1)^N Z (A_1\smallsetminus U_1 \cup \tfrac{1}{V_1}, B_1\smallsetminus V_1 \cup \tfrac{1}{U_1})  \\
& \times s_{\mu'}(-(A_1\smallsetminus U_1 \cup \tfrac{1}{V_1}))s_{\nu'}(-(B_1\smallsetminus V_1 \cup \tfrac{1}{U_1})) \sum_{ \substack{U_2\subseteq A_2 \\ V_2\subseteq B_2 \\ |U_2|=|V_2| } } (U_2V_2)^N  \\
& \times Z (A_2\smallsetminus U_2 \cup \tfrac{1}{V_2}, B_2\smallsetminus V_2 \cup \tfrac{1}{U_2}) s_{\nu'}(-(A_2\smallsetminus U_2 \cup \tfrac{1}{V_2})) s_{\mu'}(-(B_2\smallsetminus V_2 \cup \tfrac{1}{U_2})),
\end{align*}
where $\mu$ and $\nu$ are the partitions depending on $\lambda$ and defined by \eqref{eqn: mudef} and \eqref{eqn: nudef} when $\lambda=(\lambda_1,\lambda_2,\dots,\lambda_N)$. As $\lambda$ runs through all dominant weights of length $N$, the variables $\mu,\nu$ run through all partitions such that the largest part $\mu'_1$ of the conjugate $\mu'$ of $\mu$ and the largest part $\nu'_1$ of the conjugate $\nu'$ of $\nu'$ have sum $\mu_1'+\nu_1'$ that is at most $N$. It follows that
\begin{align*}
\mathcal{M}_N(A,B;0) =
& \sum_{ \substack{U_1\subseteq A_1 \\ V_1\subseteq B_1 \\ |U_1|=|V_1| } }\sum_{ \substack{U_2\subseteq A_2 \\ V_2\subseteq B_2 \\ |U_2|=|V_2| } }  (U_1U_2)^N (V_1V_2)^N  Z (A_1\smallsetminus U_1 \cup \tfrac{1}{V_1}, B_1\smallsetminus V_1 \cup \tfrac{1}{U_1})\\
& \times Z (A_2\smallsetminus U_2 \cup \tfrac{1}{V_2}, B_2\smallsetminus V_2 \cup \tfrac{1}{U_2})  \sum_{\substack{ \mu,\nu \\ \mu_1'+\nu_1'\leq N}} s_{\mu'}(-(A_1\smallsetminus U_1 \cup \tfrac{1}{V_1}))\\
& \times s_{\nu'}(-(B_1\smallsetminus V_1 \cup \tfrac{1}{U_1})) s_{\nu'}(-(A_2\smallsetminus U_2 \cup \tfrac{1}{V_2}))s_{\mu'}(-(B_2\smallsetminus V_2 \cup \tfrac{1}{U_2})).
\end{align*}
Since $\mu'$ runs over all partitions as $\mu$ does, and the same is true for $\nu'$, we may relabel $\mu'$ as $\mu$ and $\nu'$ as $\nu$. Moreover, the fact that Schur polynomials are homogeneous implies that $s_{\mu}(-C)s_{\mu}(-D)=s_{\mu}(C)s_{\mu}(D)$ for any partition $\mu$ and any finite multisets $C,D$ of complex numbers. Therefore
\begin{align}
\mathcal{M}_N(A,B;0) =
& \sum_{ \substack{U_1\subseteq A_1 \\ V_1\subseteq B_1 \\ |U_1|=|V_1| } }\sum_{ \substack{U_2\subseteq A_2 \\ V_2\subseteq B_2 \\ |U_2|=|V_2| } }  (U_1U_2)^N (V_1V_2)^N  Z (A_1\smallsetminus U_1 \cup \tfrac{1}{V_1}, B_1\smallsetminus V_1 \cup \tfrac{1}{U_1}) \notag\\
& \times Z (A_2\smallsetminus U_2 \cup \tfrac{1}{V_2}, B_2\smallsetminus V_2 \cup \tfrac{1}{U_2})  \sum_{\substack{ \mu,\nu \\ \mu_1+\nu_1\leq N}} s_{\mu}(A_1\smallsetminus U_1 \cup \tfrac{1}{V_1}) \notag\\
& \times s_{\mu}(B_2\smallsetminus V_2 \cup \tfrac{1}{U_2}) s_{\nu}(B_1\smallsetminus V_1 \cup \tfrac{1}{U_1}) s_{\nu}(A_2\smallsetminus U_2 \cup \tfrac{1}{V_2}), \label{eqn: applyBKandtwistedrecipe}
\end{align}
where the $\mu,\nu$-sum is over all pairs $\mu,\nu$ of partitions such that the largest part $\mu_1$ of $\mu$ and the largest part $\nu_1$ of $\nu$ has sum $\leq N$.

Having applied the BK Splitting and evaluated the lower twisted moments via Theorem~\ref{thm: twistedrecipe}, we now apply Theorem~\ref{thm: unitaryidentity} to deduce from \eqref{eqn: applyBKandtwistedrecipe} that
\begin{equation*}
\mathcal{M}_N(A,B;0) = \mathcal{S}_0 + \mathcal{S}_1 + \mathcal{S}_2,
\end{equation*}
where $\mathcal{S}_0$, $\mathcal{S}_1$, and $ \mathcal{S}_2$ are defined by
\begin{align}
\mathcal{S}_0: =
& \sum_{ \substack{U_1\subseteq A_1 \\ V_1\subseteq B_1 \\ |U_1|=|V_1| } }\sum_{ \substack{U_2\subseteq A_2 \\ V_2\subseteq B_2 \\ |U_2|=|V_2| } }  (U_1U_2)^N(V_1V_2)^N \notag\\
& \times Z (A_1\smallsetminus U_1 \cup \tfrac{1}{V_1}, B_1\smallsetminus V_1 \cup \tfrac{1}{U_1})Z (A_2\smallsetminus U_2 \cup \tfrac{1}{V_2}, B_2\smallsetminus V_2 \cup \tfrac{1}{U_2}) \notag\\
& \times Z (A_1\smallsetminus U_1 \cup \tfrac{1}{V_1}, B_2\smallsetminus V_2 \cup \tfrac{1}{U_2})Z (A_2\smallsetminus U_2 \cup \tfrac{1}{V_2}, B_1\smallsetminus V_1 \cup \tfrac{1}{U_1}), \label{eqn: S0def}
\end{align}
\begin{align}
\mathcal{S}_1: =
& \sum_{ \substack{U_1\subseteq A_1 \\ V_1\subseteq B_1 \\ |U_1|=|V_1| } }\sum_{ \substack{U_2\subseteq A_2 \\ V_2\subseteq B_2 \\ |U_2|=|V_2| } }  (U_1U_2)^N(V_1V_2)^N \notag\\
& \times Z (A_1\smallsetminus U_1 \cup \tfrac{1}{V_1}, B_1\smallsetminus V_1 \cup \tfrac{1}{U_1})Z (A_2\smallsetminus U_2 \cup \tfrac{1}{V_2}, B_2\smallsetminus V_2 \cup \tfrac{1}{U_2}) \notag\\
& \times \sum_{\substack{W\subseteq A_2\smallsetminus U_2 \cup \tfrac{1}{V_2} \\ Y \subseteq B_1\smallsetminus V_1 \cup \tfrac{1}{U_1}\\ |W|=|Y| >0}} (WY)^{N} Z( (A_2\smallsetminus U_2 \cup \tfrac{1}{V_2}) \smallsetminus W \cup \tfrac{1}{Y}, (B_1\smallsetminus V_1 \cup \tfrac{1}{U_1})\smallsetminus Y \cup \tfrac{1}{W}) \notag\\
& \times \sum_{a\in A_1\smallsetminus U_1 \cup \tfrac{1}{V_1} } \sum_{b\in B_2\smallsetminus V_2 \cup \tfrac{1}{U_2}} \prod_{\substack{\hat{a}\in A_1\smallsetminus U_1 \cup \tfrac{1}{V_1} \\ \hat{a}\neq a}} \left(1-\frac{\hat{a}}{a}\right)^{-1} \prod_{\substack{ \hat{b}\in B_2\smallsetminus V_2 \cup \tfrac{1}{U_2} \\ \hat{b}\neq b}} \left(1-\frac{\hat{b}}{b}\right)^{-1} \notag\\
&  \times   \sum_{\substack{Q\subseteq A_1\smallsetminus U_1 \cup \tfrac{1}{V_1} \smallsetminus \{a\} \\ R \subseteq B_2\smallsetminus V_2 \cup \tfrac{1}{U_2} \smallsetminus \{b\} \\ |Q|=|R| }}  \left( 1- \frac{ab(QR)^1}{(WY)^1} \right)^{-1} \notag\\
& \times Z((A_1\smallsetminus U_1 \cup \tfrac{1}{V_1})\smallsetminus \{a\}\smallsetminus Q \cup \tfrac{1}{R} , (B_2\smallsetminus V_2 \cup \tfrac{1}{U_2})\smallsetminus \{b\}\smallsetminus R \cup \tfrac{1}{Q}), \label{eqn: S1def}
\end{align}
and
\begin{align*}
\mathcal{S}_2: =
& \sum_{ \substack{U_1\subseteq A_1 \\ V_1\subseteq B_1 \\ |U_1|=|V_1| } }\sum_{ \substack{U_2\subseteq A_2 \\ V_2\subseteq B_2 \\ |U_2|=|V_2| } }  (U_1U_2)^N(V_1V_2)^N \notag\\
& \times Z (A_1\smallsetminus U_1 \cup \tfrac{1}{V_1}, B_1\smallsetminus V_1 \cup \tfrac{1}{U_1})Z (A_2\smallsetminus U_2 \cup \tfrac{1}{V_2}, B_2\smallsetminus V_2 \cup \tfrac{1}{U_2}) \notag\\
& \times  \sum_{\substack{W\subseteq A_2\smallsetminus U_2 \cup \tfrac{1}{V_2} \\ Y \subseteq B_1\smallsetminus V_1 \cup \tfrac{1}{U_1}\\ |W|=|Y| }}   Z( (A_2\smallsetminus U_2 \cup \tfrac{1}{V_2})\smallsetminus W \cup \tfrac{1}{Y}, (B_1\smallsetminus V_1 \cup \tfrac{1}{U_1})\smallsetminus Y \cup \tfrac{1}{W}) \\
& \times \sum_{a\in A_1\smallsetminus U_1 \cup \tfrac{1}{V_1} } \sum_{b\in B_2\smallsetminus V_2 \cup \tfrac{1}{U_2}} \prod_{\substack{\hat{a}\in A_1\smallsetminus U_1 \cup \tfrac{1}{V_1} \\ \hat{a}\neq a}} \left(1-\frac{\hat{a}}{a}\right)^{-1} \prod_{\substack{ \hat{b}\in B_2\smallsetminus V_2 \cup \tfrac{1}{U_2} \\ \hat{b}\neq b}} \left(1-\frac{\hat{b}}{b}\right)^{-1} \\
&  \times   \sum_{\substack{Q\subseteq A_1\smallsetminus U_1 \cup \tfrac{1}{V_1} \smallsetminus \{a\} \\ R \subseteq B_2\smallsetminus V_2 \cup \tfrac{1}{U_2} \smallsetminus \{b\} \\ |Q|=|R| }}  ( ab)^N(QR )^{N}\left(  1- \frac{(WY)^1} {ab(QR)^1} \right)^{-1} \\
& \times Z((A_1\smallsetminus U_1 \cup \tfrac{1}{V_1})\smallsetminus \{a\}\smallsetminus Q \cup \tfrac{1}{R} , (B_2\smallsetminus V_2 \cup \tfrac{1}{U_2})\smallsetminus \{b\}\smallsetminus R \cup \tfrac{1}{Q}),
\end{align*}
respectively. If $C,D_1,D_2$ are finite sets of complex numbers with $D_1$ disjoint from $D_2$, then the definition \eqref{eqn: Zdef} implies that $Z(C,D_1)Z(C,D_2)=Z(C,D_1\cup D_2)$ and $Z(D_1,C)Z(D_2,C)=Z(D_1\cup D_2,C)$. Thus, since $\{A_1,A_2\}$ is a partition of the set $A$ and $\{B_1,B_2\}$ is a partition of the set $B$, we may make a change of variables $U=U_1\cup U_2$ and $V=V_1\cup V_2$ in \eqref{eqn: S0def} to deduce that
\begin{equation*}
\mathcal{S}_0 = \sum_{\substack{ U \subseteq A \\ V \subseteq B \\ |U| =|V| \\ |U\cap A_1| = |V\cap B_1| }} (UV)^N Z (A \smallsetminus U  \cup \tfrac{1}{V }, B \smallsetminus V  \cup \tfrac{1}{U }).
\end{equation*}
Thus, $\mathcal{S}_0$ is exactly the sum of the terms on the right-hand side of \eqref{eqn: recipe} that have $|U\cap A_1|=|V\cap B_1|$.

To finish the proof of \eqref{eqn: recipe}, it is left to show that $\mathcal{S}_1+\mathcal{S}_2$ equals the sum of the terms on the right-hand side of \eqref{eqn: recipe} that have $|U\cap A_1|\neq |V\cap B_1|$. We first show that $\mathcal{S}_1+\mathcal{S}_2$ equals a sum of terms of the form
\begin{equation*}
(UV)^N f_{U,V}
\end{equation*}
with $f_{U,V}$ a rational function of the elements of $A\cup B$ that is independent of $N$ and with $U,V$ sets such that $U\subseteq A$, $V\subseteq B$, $|U|=|V|$, and $|U\cap A_1|\neq |V\cap B_1|$. We see from \eqref{eqn: S1def} that $\mathcal{S}_1$ is a sum of terms, each of which is of the form
\begin{equation*}
(U_1U_2)^N (V_1V_2)^N (WY)^N f_{U_1,U_2,V_1,V_2,W,Y}
\end{equation*}
with $f_{U_1,U_2,V_1,V_2,W,Y}$ a rational function of the elements of $A\cup B$ that is independent of $N$. If $U_1$, $U_2$, $V_1$, and $V_2$ are sets that satisfy the conditions in \eqref{eqn: S1def}, then $(U_1U_2)^1$ is a product of factors, each of which is an element of $A$, and similarly $(V_1V_2)^1$ is a product of factors, each of which is an element of $B$. Moreover, $(U_1U_2)^1$ has the same number of factors as $(V_1V_2)^1$, and the number of factors in $(U_1U_2)^1$ that belong to $A_1$ is equal to the number of factors in $(V_1V_2)^1$ that belong to $B_1$. If $Y \subseteq B_1\smallsetminus V_1 \cup \frac{1}{U_1}$, then multiplying an element of $Y$ to the product $(U_1U_2)^1$ results to a product that either has one more factor belonging to $B_1$ or has one less factor belonging to $U_1$ and thus $A_1$. Similarly, if $W\subseteq A_2\smallsetminus U_2 \cup \frac{1}{V_2}$, then multiplying an element of $W$ to the product $(V_1V_2)^1$ results to a product that either has one more factor belonging to $A_2$ or has one less factor belonging to $V_2$ and thus $B_2$. Hence, for $U_1,U_2,V_1,V_2,W,Y$ satisfying the conditions in \eqref{eqn: S1def}, the number of factors in the product $(U_1U_2)^1(V_1V_2)^1(WY)^1$ that belong to $A$ is equal to the number of those that belong to $B$ because $|W|=|Y|$, and the number of factors that belong to $A_1$ is less than the number of those that belong to $B_1$ because $Y$ is nonempty. In other words, $\mathcal{S}_1$ equals a sum of terms of the form $(UV)^N f_{U,V}$ with $f_{U,V}$ a rational function of the elements of $A\cup B$ that is independent of $N$ and with $U,V$ sets such that $U\subseteq A$, $V\subseteq B$, $|U|=|V|$, and $|U\cap A_1|<|V\cap B_1|$. A similar argument shows that $\mathcal{S}_2$ equals a sum of terms of the form $(UV)^N f_{U,V}$ with $f_{U,V}$ a rational function of the elements of $A\cup B$ that is independent of $N$ and with $U,V$ sets such that $U\subseteq A$, $V\subseteq B$, $|U|=|V|$, and $|U\cap A_1|>|V\cap B_1|$. 

In summary, we have shown so far that if $N$ is a positive integer, $\{A_1,A_2\}$ is any partition of $A$, and $\{B_1,B_2\}$ is any partition of $B$, then
\begin{equation}\label{eqn: applyBK2}
\mathcal{M}_N(A,B;0)= \sum_{\substack{U \subseteq A \\ V\subseteq B \\ |U|=|V|}} (UV)^N f_{U,V},
\end{equation}
where each $f_{U,V}$ is a rational function of the elements of $A \cup B$ that is independent of $N$ such that if $|U\cap A_1| =|V\cap B_1|$, then
\begin{equation}\label{eqn: applyBK3}
f_{U,V} = Z (A \smallsetminus U  \cup \tfrac{1}{V }, B \smallsetminus V  \cup \tfrac{1}{U }).
\end{equation}
To show that this is also true if $|U\cap A_1| \neq |V\cap B_1|$, suppose that $U_0 \subseteq A$ and $V_0\subseteq B$ with $|U_0|=|V_0|$ and $|U_0\cap A_1| \neq |V_0\cap B_1|$. Then there exist partitions $\{A_1',A_2'\}$ of $A$ and $\{B_1',B_2'\}$ of $B$ such that $|U_0\cap A_1'|=|V_0\cap B_1'|$. Since \eqref{eqn: applyBK2} with \eqref{eqn: applyBK3} holds for arbitrary partitions of $A$ and $B$ into two nonempty subsets each, it follows that, for all positive integers $N$,
\begin{equation}\label{eqn: applyBK4}
\mathcal{M}_N(A,B;0)= \sum_{\substack{U \subseteq A \\ V\subseteq B \\ |U|=|V|}} (UV)^N g_{U,V},
\end{equation}
where each $g_{U,V}$ is a rational function of the elements of $A \cup B$ that is independent of $N$ such that if $|U\cap A_1'| =|V\cap B_1'|$, then
\begin{equation}\label{eqn: applyBK5}
g_{U,V} = Z (A \smallsetminus U  \cup \tfrac{1}{V }, B \smallsetminus V  \cup \tfrac{1}{U }).
\end{equation}
In particular, \eqref{eqn: applyBK5} holds when $U=U_0$ and $V=V_0$ because $|U_0\cap A_1'|=|V_0\cap B_1'|$. From \eqref{eqn: applyBK2} and \eqref{eqn: applyBK4}, we deduce that
\begin{equation*}
\sum_{\substack{U \subseteq A \\ V\subseteq B \\ |U|=|V|}} (UV)^N f_{U,V} = \sum_{\substack{U \subseteq A \\ V\subseteq B \\ |U|=|V|}} (UV)^N g_{U,V}
\end{equation*}
for all positive integers $N$. Clearing denominators and applying Lemma~\ref{lem: linearindependence}, we see that
\begin{equation*}
f_{U,V} = g_{U,V}
\end{equation*}
for all $U,V$. From this and \eqref{eqn: applyBK5} with $U=U_0$ and $V=V_0$, we arrive at
\begin{equation*}
f_{U_0,V_0} = Z (A \smallsetminus U_0  \cup \tfrac{1}{V_0 }, B \smallsetminus V_0  \cup \tfrac{1}{U_0 }).
\end{equation*}
Since $U_0$ and $V_0$ are arbitrary, it follows that \eqref{eqn: applyBK3} holds for all $U,V$ with $|U \cap A_1| \neq |V \cap B_1|$, and hence it holds for all $U,V$. From this and \eqref{eqn: applyBK2}, we conclude that
\begin{equation*}
\mathcal{M}_N(A,B;0)= \sum_{\substack{ U \subseteq A \\ V \subseteq B \\ |U| =|V|}} (UV)^N Z (A \smallsetminus U  \cup \tfrac{1}{V }, B \smallsetminus V  \cup \tfrac{1}{U }).
\end{equation*}
This proves \eqref{eqn: recipe}.

\appendix

\section{Proofs of Lemmas~\ref{lem: schurorthogonal} and \ref{lem: bumpgamburdlemma} }\label{sec: appendix}

\begin{proof}[Proof of Lemma~\ref{lem: schurorthogonal}]
This is a generalization of a well-known proof (see, for example, Section~15.1 of \cite{conrey}) of the fact that Schur polynomials are orthogonal over $U(N)$. The said proof applies verbatim to generalized Schur functions. The Weyl Integration Formula (see, for example, Proposition~17.4 of \cite{bump}) and the definition \eqref{eqn: genschurdef} imply that if $\lambda=(\lambda_1,\lambda_2,\dots,\lambda_N)$ and $\rho=(\rho_1,\rho_2,\dots,\rho_N)$ are dominant weights, then
\begin{equation*}
\mathcal{J}_N(\lambda,\rho) = \frac{1}{(2\pi)^N N!}\int_{[0,2\pi]^N} \det\Big(e^{i\theta_j(N-k+\lambda_k)} \Big)_{1\leq j,k\leq N} \det\Big(e^{-i\theta_j(N-k+\rho_k)} \Big)_{1\leq j,k\leq N} \,d\theta_1\cdots d\theta_N.
\end{equation*}
It follows from this and the Andr\'{e}ief identity~\cite{andreief} (for a more current reference, see Chapter~11 of \cite{lnv}, for instance) that
\begin{equation*}
\mathcal{J}_N(\lambda,\rho) =  \det\Bigg( \frac{1}{2\pi }\int_{0}^{2\pi} e^{i\theta (N-j+\lambda_j)} e^{-i\theta (N-k+\rho_k)}  \,d\theta \Bigg)_{1\leq j,k\leq N}.
\end{equation*}
Hence, if we let $\delta(x)=1$ for $x=1$ and $\delta(x)=0$ for all other $x$, then
\begin{equation}\label{eqn: schurorthogonal1}
\mathcal{J}_N(\lambda,\rho) =  \det\Big( \delta(k-j+\lambda_j-\rho_k)\Big)_{1\leq j,k\leq N}.
\end{equation}
Since $\lambda$ and $\rho$ are dominant weights, the quantity $k-j+\lambda_j-\rho_k$ is strictly increasing in $k$ for fixed $j$ and strictly decreasing in $j$ for fixed $k$. Thus, if $\lambda_1<\rho_1$, then all the entries in column $1$ (the entries with $k=1$) of the matrix in \eqref{eqn: schurorthogonal1} are zero, which implies that the determinant is zero. Similarly, if $\lambda_1>\rho_1$, then all the entries in row $1$ (the entries with $j=1$) are zero and the determinant is zero. Therefore $\mathcal{J}_N(\lambda,\rho)=0$ unless $\lambda_1=\rho_1$. In this latter case, all the entries that are in either row $1$ or column $1$ are zero except for the $1,1$ entry, which equals $1$, and the cofactor expansion of the determinant implies
\begin{equation*}
\mathcal{J}_N(\lambda,\rho) =  \det\Big( \delta(k-j+\lambda_j-\rho_k)\Big)_{2\leq j,k\leq N}.
\end{equation*}
By a similar argument, the latter determinant is zero unless $\lambda_2=\rho_2$. Continuing in this way, we see that $\mathcal{J}_N(\lambda,\rho)=0$ unless $\lambda=\rho$, in which case $\mathcal{J}_N(\lambda,\rho)=1$.
\end{proof}

\begin{proof}[Proof of Lemma~\ref{lem: bumpgamburdlemma}]
The Littlewood-Richardson rule (see, for example, (9.2) in Chapter~I of \cite{macdonald}) states that if $\lambda,\mu,\nu$ are partitions, then $c^{\lambda}_{\mu,\nu}$ equals the number of Littlewood-Richardson tableaux of shape $\lambda/\mu$ and content $\nu$. Since there are no skew tableaux of shape $\langle N^L \rangle/\mu$ unless $\mu \subset \langle N^L \rangle$, it follows that $c^{\langle N^L \rangle}_{\mu,\nu}$ equals $0$ unless the length of $\mu$ is at most $L$. Similarly, since $c^{\lambda}_{\mu,\nu}=c^{\lambda}_{\nu,\mu}$, the coefficient $c^{\langle N^L \rangle}_{\mu,\nu}$ equals $0$ unless the length of $\nu$ is at most $L$. Suppose now that $\mu$ and $\nu$ each have length at most $L$. We claim that there is exactly one Littlewood-Richardson tableau of shape $\langle N^L \rangle/\mu$. Indeed, the lattice permutation property together with the weakly increasing rows property forces the tableau to have only $1$'s at the topmost nonempty row. Immediately below that, the lattice permutation property, weakly increasing rows property, and strictly increasing columns property forces the tableau to have, from right to left, all $2$'s at the squares below the $1$'s of the top row, and then all $1$'s after that. Next, at the row immediately below that, from right to left, the tableau must have all $3$'s at the squares below the $2$'s, all $2$'s at the squares below the $1$'s, and all $1$'s at the squares with no squares above them. Continuing this process, we see that there is exactly one Littlewood-Richardson tableau of shape $\langle N^L \rangle/\mu$. Moreover, counting the $1$'s, $2$'s, $3$'s, etc., in this Littlewood-Richardson tableau shows that the tableau has content exactly
$$
(N-\mu_L, N-\mu_{L-1},\dots, N-\mu_1,0,0,\dots).
$$
We conclude that $c^{\langle N^L \rangle}_{\mu,\nu}$ is equal to $0$ unless this content equals $\nu$, in which case $c^{\langle N^L \rangle}_{\mu,\nu}=1$.
\end{proof}

\end{document}